\documentclass[letterpaper]{amsart}

\usepackage[initials,nobysame]{amsrefs}
\usepackage{amssymb, mathtools, mathrsfs, enumitem, hyperref}
\usepackage{pgfplots}\pgfplotsset{compat=1.18}
\BibSpec{collection.article}{%
	+{}  {\PrintAuthors}				{author}
	+{,} { \textit}                  	{title}
	+{.} { }                         	{part}
	+{:} { \textit}                  	{subtitle}
	+{,} { \PrintContributions}      	{contribution}
	+{,} { \PrintConference}         	{conference}
	+{}  {\PrintBook}                	{book}
	+{,} { }                         	{booktitle}
	+{,} { }						 	{series}
	+{,} { }						 	{publisher}
	+{,} { }						 	{address}
	+{,} { \PrintDateB}              	{date}
	+{,} { pp.~}                     	{pages}
	+{,} { }                         	{status}
	+{,} { \PrintDOI}                	{doi}
	+{,} { available at \eprint}     	{eprint}
	+{}  { \parenthesize}            	{language}
	+{}  { \PrintTranslation}        	{translation}
	+{;} { \PrintReprint}            	{reprint}
	+{.} { }                         	{note}
	+{.} {}                          	{transition}
	+{}  {\SentenceSpace \PrintReviews} {review}
}

\theoremstyle{plain}
\newtheorem{theorem}{Theorem}
\newtheorem{lemma}[theorem]{Lemma}
\newtheorem{proposition}[theorem]{Proposition}

\newtheorem*{conjecture}{Conjecture}

\theoremstyle{definition}
\newtheorem{definition}[theorem]{Definition}

\theoremstyle{remark}

\numberwithin{equation}{section}
\numberwithin{theorem}{section}
\renewcommand{\bar}{\overline}

\newcommand{\C}{\mathbb C}
\newcommand{\D}{\mathbb D}

\newcommand{\N}{\mathbb N}

\DeclareMathOperator{\dist}{dist}
\DeclareMathOperator{\diam}{diam}

\DeclareMathOperator{\Mod}{Mod}

\title{Koebe uniformization, John domains, and transboundary modulus}

\author{Christina Karafyllia}
\address{Department of Mathematics, University of Western Macedonia, Kastoria, 52100, Greece}
\email{chkarafyllia@uowm.gr}   

\author{Dimitrios Ntalampekos}
\address{Department of Mathematics, Aristotle University of Thessaloniki, Thessaloniki, 54152, Greece.}
\thanks{D.~N.\ is supported by the European Union (ERC, GRComPaS, Grant Agreement no.\ 101214615). Views and opinions expressed are however those of the author(s) only and do not necessarily reflect those of the European Union or the European Research Council. Neither the European Union nor the granting authority can be held responsible for them.}
\email[Corresponding author]{dntalam@math.auth.gr}

\date{\today}
\keywords{circle domain, uniformization, Koebe conjecture, John domain, transboundary modulus}
\subjclass[2020]{Primary 30C20; Secondary 30C35}

\begin{document}

	\begin{abstract}
		We prove Koebe's conjecture for John domains by showing that every planar John domain is conformally equivalent to a circle domain. Our proof is based on a geometric property of John domains, which we prove to be sufficient for uniformization by a circle domain. As a consequence, we obtain an alternative proof of Koebe's conjecture for inner uniform and Gromov hyperbolic domains, recovering results established in earlier work of the authors. Finally, we derive a necessary and sufficient condition for a domain to satisfy Koebe's conjecture by using Schramm's transboundary modulus and well-behaved exhaustions of circle domains constructed by the second author and Rajala. 
	\end{abstract}

\maketitle

\section{Introduction}

The purpose of this work is to establish Koebe's conjecture for the class of John domains in the plane $\C$.

\begin{theorem}\label{theorem:intro}
Every John domain is conformally equivalent to a circle domain. 
\end{theorem}

\subsection{Background}\label{section:background}
A domain in the Riemann sphere $\widehat \C=\C\cup\{\infty\}$ whose boundary components are points or circles is called a \textit{circle domain}. In 1908, Koebe \cite{Koebe:Kreisnormierungsproblem} posed the following deep conjecture, known as the Kreisnormierungsproblem.

\begin{conjecture}Every domain in the Riemann sphere is conformally equivalent to a circle domain. 
\end{conjecture}

Koebe \cite{Koebe:FiniteUniformization} established the conjecture for finitely connected domains. More than half a century later, He and Schramm \cite{HeSchramm:Uniformization} proved, in a celebrated work, that if $\Omega$ is a countably connected domain in the Riemann sphere, then there exists a conformal map $f$ from $\Omega$ onto a circle domain. In addition, $f$ is unique up to postcomposition with M\"{o}bius transformations. This remains the most general result known to date concerning Koebe's conjecture, while the conjecture in full generality is still open.

A fruitful approach to the conjecture is to verify it for various classes of domains satisfying some good geometric conditions. The techniques developed for the treatment of these special cases provide valuable insights into the general problem. In the present work we focus on the class of John domains.

John domains were introduced by F.\ John \cite{John:domains} and subsequently studied by Martio and Sarvas \cite{MartioSarvas:uniform}. Roughly speaking, a planar domain is a John domain if any two points in the domain can be connected by a curve that stays quantitatively away from the boundary. In particular, John domains cannot have outward cusps, although they may have inward cusps. By their original definition, John domains are bounded. However, Nakki and V\"ais\"al\"a \cite{NakkiVaisala:john} extended the definition to include unbounded domains. In this paper, we adopt the definition given in \cite{NakkiVaisala:john}.

\begin{definition}\label{definition:john}
Let $\Omega\subset \C$ be a domain and let $L\geq 1$.  We say that $\Omega$ is an \textit{$L$-John domain} if for every pair of distinct points $z,w\in \Omega$ there exists a curve $\gamma\colon [0,1]\to \Omega$ such that  $\gamma(0)=z$, $\gamma(1)=w$, and
\begin{align}\label{def:john}
\min\{\ell (\gamma|_{[0,t]}), \ell(\gamma|_{[t,1]} )\} \leq L \dist(\gamma(t),\partial \Omega)\,\,\, \text{for all $t\in [0,1]$.}
\end{align}
We say that $\Omega$ is a John domain if it is an $L$-John domain for some $L\ge 1$. 
\end{definition}

If, in addition to \eqref{def:john}, we require the curve $\gamma$ to satisfy
\begin{align}\label{def:uniform}
\ell (\gamma)\le L|z-w|,
\end{align}
then $\Omega$ is a \textit{uniform domain}. These domains were introduced independently by Martio--Sarvas \cite{MartioSarvas:uniform} and Jones \cite{Jones:uniform} and appear in several problems of geometric function theory, since they enjoy several of the favorable properties of the unit disk. If, instead of requiring \eqref{def:uniform}, we impose the weaker condition
\[\ell (\gamma)\le L \inf_{\beta} \ell (\beta),\]
where the infimum is taken over all paths $\beta$ in $\Omega$ connecting $z$ to $w$, then $\Omega$ is called an \textit{inner uniform domain}. Thus, among the classes of uniform, inner uniform, and John domains, the latter is the broadest.

Herron and Koskela \cite{HerronKoskela:QEDcircledomains} verified Koebe's conjecture for the class of {uniform domains}. More recently, the conjecture was established in \cite{KarafylliaNtalampekos:gromov_hyperbolic} by the current authors for the class of inner uniform domains. An important consequence of that work is that a domain is \textit{Gromov hyperbolic} if and only if it is the conformal image of a uniform domain, thereby verifying a conjecture of Bonk--Heinonen--Koskela \cite{BonkHeinonenKoskela:gromov_hyperbolic}. Thus, the present work provides an alternative and completely different proof of the main result of \cite{KarafylliaNtalampekos:gromov_hyperbolic}. Although the definitions of inner uniform and John domains have similarities, the geometry of John domains can be considerably wilder. Thus, the techniques developed in \cite{KarafylliaNtalampekos:gromov_hyperbolic} do not apply, as we discuss below. 

Instead, our approach relies on the notion of transboundary modulus, introduced by Schramm in his remarkable work \cite{Schramm:transboundary}. This notion has become a standard tool in modern complex analysis and it is one of the main tools used in the present work. Schramm used transboundary modulus to verify Koebe's conjecture for \textit{cofat domains}. A measurable set $P\subset \widehat \C$ is $\tau$-fat for some $\tau>0$ if, for every ball $B$ of spherical radius $r$ that does not contain $P$, the spherical area of $B\cap P$ is at least $\tau r^2$. A domain $\Omega\subset \widehat{\C}$ is cofat if there exists $\tau>0$ such that each component of $\widehat \C\setminus \Omega$ is $\tau$-fat. More recently, Esmayli and Rajala \cite{EsmayliRajala:quasitripod} generalized Schramm's result and verified the conjecture for the class of \textit{cospread domains}, whose complementary components are allowed to have area zero. Note that the complement of a line segment is a John domain, but it is neither a cofat nor a cospread domain.

\subsection{Uniqueness}
If there exists a conformal map $f$ from a domain $\Omega$ onto a circle domain one can ask whether it is unique. Namely, if $g$ is another conformal map from $\Omega$ onto a circle domain, is $f\circ g^{-1}$ the restriction of a M\"obius transformation?

The answer is affirmative for finitely connected domains, by a theorem of Koebe \cite{Koebe:FiniteUniformization}, and for countably connected domains, due to He and Schramm \cite{HeSchramm:Uniformization}. The uniqueness problem, however, is distinct from the existence problem and requires different techniques. These include the use of the Schottky groups generated by reflections in circles, as well as removability results from the theory of quasiconformal maps. For instance, although the existence in the case of uniform domains was established by Herron and Koskela in \cite{HerronKoskela:QEDcircledomains}, uniqueness was proved only three decades later, by the author and Younsi \cite{NtalampekosYounsi:rigidity}. More precisely, they showed that if a circle domain is a uniform domain or a John domain (or, more generally, satisfies certain geometric conditions), then any conformal map from it onto another circle domain is the restriction of a M\"obius transformation.

Therefore, assuming that there exists a conformal map $f$ from a domain $\Omega$ onto a circle domain $D$, in order to prove the uniqueness of $f$, it suffices to show that $D$ has some nice geometric properties, as the ones appearing in \cite{NtalampekosYounsi:rigidity}. We note, however, that in the case of a John domain $\Omega$ this is not necessarily true. For example, consider a domain $\Omega$ whose boundary is contained in $[0,1]$ and there exist pairs of components $E_n,F_n\subset \partial \Omega$, $n\in \N$, whose relative distances 
$$\Delta(E_n,F_n)=\frac{\dist(E_n,F_n)}{\min\{\diam E_n,\diam F_n\}}$$
converge to $0$ as $n\to\infty$. Then $\Omega$ can be conformally mapped to a circle domain $D$ due to symmetry \cite{Koebe:symmetric}. Using conformal modulus it can be proved that the circles $E_n',F_n'$ of $\partial D$ corresponding to $E_n,F_n$, respectively, must also have relative distances converging to $0$. Thus, $E_n',F_n'$ tend to be tangent to each other. This shows that the circle domain $D$ tends to have an outward cusp between $E_n'$ and $F_n'$ so it is not a uniform, an inner uniform or a John domain.
 
We do not discuss the uniqueness problem further and we leave it as an open problem for John domains. More background on the uniqueness and recent results can be found in \cites{HeSchramm:Rigidity, Younsi:RemovabilityRigidityKoebe, Ntalampekos:rigidity_cned, Rajala:rigidity}. See also the survey paper \cite{Ntalampekos:uniformization_survey}.

\subsection{Uniformization by exhaustion}\label{section:exhaustion_intro}

The proof of our main theorem for the uniformization of John domains utilizes \textit{exterior approximations}. That is, we consider finitely connected domains $\Omega_j\supset \Omega$, $j\in \N$, we uniformize $\Omega_j$ conformally by a circle domain $D_j$ and then pass to the limit; see Section \ref{section:strategy} for further details. The same approach is used in several works related to the conjecture \cites{HeSchramm:Uniformization, Schramm:transboundary, EsmayliRajala:quasitripod, KarafylliaNtalampekos:gromov_hyperbolic}. There is a competing approach using exhaustions of a domain, first studied by Rajala \cite{Rajala:koebe}, that we describe below. In the current work we reconcile the two approaches, since almost all of our technical results in Section \ref{section:uniformization} can be equally applied to exhaustions and to exterior approximations.

Let $\Omega\subset \widehat \C$ be a domain. We denote by $\mathcal C(\Omega)$ the collection of components of $\widehat \C\setminus \Omega$. An \textit{exhaustion} of $\Omega$ is a sequence of domains $\{\Omega_j\}_{j\in \N}$ such that, for each $j\in \N$, $\partial \Omega_j$ is the union of finitely many disjoint Jordan curves in $\Omega$, $\Omega_j\subset \Omega_{j+1}\subset \Omega$, and $\bigcup_{j\in \N}\Omega_j=\Omega$. Observe that for each $q\in \mathcal C(\Omega)$ and for each $j\in \N$ there exists a unique component $q_j(q)\in \mathcal C(\Omega_j)$ such that $q\subset q_j(q)$.

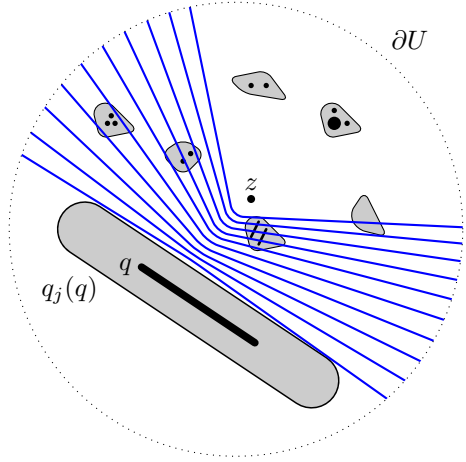
\begin{figure}
	\centering
	\begin{tikzpicture}
		\begin{scope}
		\clip (0,0) circle (3.0cm);		
		\draw[line width=0.74cm, line cap=round](-2,0)--(1,-2);
		\draw[line width=0.7cm, line cap=round, color=black!20](-2,0)--(1,-2);
		\draw[line width=0.1cm, line cap=round](-1.25,-0.5)--(0.25,-1.5);
		\node[left] at (-1.25,-0.5) {$q$};
		\node[below] at (-2.2,-0.5) {$q_j(q)$};
		
		\draw[rounded corners, fill=black!20] (-1,1)--(-.7,.7)--(-0.4,1)--(-.7,1.2)--cycle;
		\fill(-0.7,0.9) circle (1pt);		
		\fill(-0.6,1) circle (1pt);
				
		\draw[rounded corners=3pt, fill=black!20, shift={(1,-1)}] (-0.9,1.1)--(-.8,.7)--(-0.3,.8)--(-.7,1.2)--cycle;
		\draw[line width=1pt,shift={(1,-1)}, line cap=round](-0.7,0.8)--(-0.6,1);  		
		\draw[line width=1pt,shift={(1,-1)}, line cap=round](-0.8,0.9)--(-0.7,1.1); 	
		
		\draw[rounded corners=3pt, fill=black!20, shift={(1,1)}] (-1.1,1)--(-.9,.8)--(-0.3,.7)--(-.6,1.1)--cycle;
		\fill[shift={(1,1)}](-0.8,0.9) circle (1pt);		
		\fill[shift={(1,1)}](-0.6,.9) circle (1pt);
		
		\draw[rounded corners=3pt, fill=black!20, shift={(-1,0.5)}] (-0.9,1.1)--(-.8,.7)--(-0.3,.8)--(-.7,1.2)--cycle;
		\fill[shift={(-1,0.5)}](-0.7,0.9) circle (1pt);		
		\fill[shift={(-1,0.5)}](-0.6,.9) circle (1pt);
		\fill[shift={(-1,0.5)}](-0.65,1) circle (1pt);
		
		\draw[rounded corners=3pt, fill=black!20, shift={(2,0.5)}] (-0.9,1.1)--(-.8,.7)--(-0.3,.8)--(-.7,1.2)--cycle;
		\fill[shift={(2,0.5)}](-0.7,0.9) circle (2.5pt);
		\fill[shift={(2,0.5)}](-0.53,0.9) circle (1pt);
		\fill[shift={(2,0.5)}](-0.7,1.07) circle (1pt);
		
		\draw[rounded corners=3pt, fill=black!20, shift={(2.5,-0.8)}] (-1,1)--(-.9,.8)--(-0.5,.7)--(-.8,1.3)--cycle;
		
		\fill (0.2,0.4) circle (1.5pt) node[above] {$z$};

		\draw[blue,thick,rounded corners] (-4.8,2.2)--(-0.6,-0.4)--(4.2,-3.8);
		\draw[blue,thick,rounded corners] (-4,4,2.6)--(-0.55,-0.35)--(4.6,-3.4);
		\draw[blue,thick,rounded corners] (-4,3)--(-0.5,-0.3)--(5,-3);
		\draw[blue,thick,rounded corners] (-3.6,3.4)--(-0.44,-0.24)--(5.4,-2.6);
		\draw[blue,thick,rounded corners] (-3.2,3.8)--(-0.37,-0.17)--(5.8,-2.2);
		\draw[blue,thick,rounded corners] (-2.8,4.2)--(-0.3,-0.1)--(6.2,-1.8);
		\draw[blue,thick,rounded corners] (-2.4,4.6)--(-0.24,-0.04)--(6.6,-1.4);
		\draw[blue,thick,rounded corners] (-2,5)--(-0.17,0.03)--(7,-1);
		\draw[blue,thick,rounded corners] (-1.6,5.4)--(-0.1,0.1)--(7.4,-0.6);
		\draw[blue,thick,rounded corners] (-1.2,5.8)--(-0.04,0.17)--(7.8,-0.2);			\end{scope}
		
		\draw[dotted] (0,0) circle (3cm);
		\node at (2.3,2.5) {$\partial U$};
	\end{tikzpicture}
	\caption{An exhaustion of a domain and some curves lying in the family $\mathcal G^*(z,q_j(q);U)$ that appears in Theorem \ref{theorem:exhaustion}.}\label{fig:exhaustion}
\end{figure}

According to a recent result of the second author and Rajala \cite{NtalampekosRajala:exhaustion}*{Theorem 1.1}, if a domain $\Omega\subset \widehat \C$ satisfies Koebe's conjecture then there exists an exhaustion $\{\Omega_j\}_{j\in \N}$ such that if $f_j\colon \Omega_j\to D_j$, $j\in \N$, is a sequence of conformal maps given by Koebe's uniformization theorem, normalized appropriately, then $\{f_j\}_{j\in \N}$ converges locally uniformly in $\Omega$ to a conformal map $f$ from $\Omega$ onto a circle domain. We refine this result by giving the precise property of the exhaustion $\{\Omega_j\}_{j\in \N}$ that is needed for Koebe's conjecture to be true.

\begin{theorem}\label{theorem:exhaustion}
A domain $\Omega\subset \widehat \C$ satisfies Koebe's conjecture if and only if there exist an exhaustion $\{\Omega_j\}_{j\in \N}$ of $\Omega$ and a Jordan region $U\subset \widehat \C$ that contains $\widehat \C\setminus \Omega$ such that for every $ q\in \mathcal C(\Omega)$ we have
$$\lim_{\substack{z\to  q\\z\in \Omega}} \limsup_{j\to\infty} \Mod_{\Omega_j} \mathcal G^*(z,q_j(q); U)=0,$$
where $\mathcal G^*(z,q_j(q);U)$ is the family of curves in $\pi_{\Omega_j}(\bar U)$ with endpoints on $\partial U$ that separate $z$ and $q_j(q)$ in $\pi_{\Omega_j}(U)$. 
\end{theorem}

See Figure \ref{fig:exhaustion} for an illustration of the curve family $\mathcal G^*(z,q_j(q);U)$. Here, $\Mod_{\Omega_j}$ denotes transboundary modulus with respect to the domain $\Omega_j$ and $\pi_{\Omega_j}$ is the projection from $\Omega_j$ onto the topological sphere $\hat \Omega_j$ arising by identifying all points of each complementary component of $\Omega_j$. See Section \ref{section:modulus} for details. The proof of the theorem is given in Section \ref{section:exhaustion}.

\subsection{Proof strategy}\label{section:strategy}

We first discuss the proof of the uniformization of inner uniform domains from \cite{KarafylliaNtalampekos:gromov_hyperbolic}. The proof consists of three parts. First, given an inner uniform domain $\Omega$ it is shown that $\Omega$ can be approximated by finitely connected inner uniform domains $\Omega_j$, $j\in \N$, all with the same constant. This is the most technical step and requires the establishment of several geometric properties of inner uniform domains. Second, it is shown that if an inner uniform domain $\Omega_j$ is conformally mapped onto a circle domain $D_j$, then $D_j$ is a uniform domain, quantitatively. Third, the proof invokes a crucial property of uniform circle domains: if $\{D_j\}_{j\in \N}$, is a sequence of circle domains that are uniform domains with the same constant, then any Carath\'eodory limit of $\{D_j\}_{j\in \N}$ is also a uniform circle domain.

However, this strategy fails for John domains. In fact, it is plausible that the analogue of the first step, which is the most technical part in \cite{KarafylliaNtalampekos:gromov_hyperbolic}, can be established more easily for John domains, since their definition is less restrictive than the one of inner uniform domains. However, the second step fails as illustrated by the example of the previous section. Also, the analogue of the third step fails quite dramatically: the Carath\'eodory limit of a sequence of John circle domains need not be a circle domain! For example, for $j\in \N$ let $\Omega_j$ be the domain whose boundary is the finite set $\{k/j: k=1,\dots,j\}$. The domains $\Omega_j$, $j\in \N$, are John domains with a uniform constant.  Nevertheless, $\{\Omega_j\}_{j\in \N}$ converges in the Carath\'eodory sense to $\widehat \C\setminus [0,1]$. Therefore, the proof strategy of \cite{KarafylliaNtalampekos:gromov_hyperbolic} has to be abandoned.

A priori it is not clear whether the transboundary modulus method of Schramm \cite{Schramm:transboundary} can be applied to John domains, since their complementary components can have area zero. However, the recent work of Esmayli--Rajala \cite{EsmayliRajala:quasitripod} manages to extend Schramm's method to cospread domains, whose complementary components contain quasitripods in all locations and scales. A particularly important geometric property of these domains is the \textit{packing condition}, stating that the number of boundary components intersecting a ball of radius $r$ and having diameter at least $r$ is uniformly bounded. It is intuitively clear that this condition holds for John domains, since they are not allowed to have narrow passages. Thus, a promising idea is to adapt Schramm's transboundary modulus method to John domains.

In the first part of the paper, in Section \ref{section:john}, we establish geometric properties of John domains that are relevant for applying Schramm's method. Specifically, in Theorem \ref{theorem:bounded_turning} we prove that the complementary components of an $L$-John domain $\Omega$ have bounded turning, quantitatively. We use this to show in Proposition \ref{prop:packing} that the packing condition is satisfied: the number of complementary components of $\Omega$ that intersect a ball $\D(z_0,r)$ and have diameter at least $r/a$ is bounded above by a constant $N(L,a)$. 

Next, we prove in Proposition \ref{prop:shadows} that although the complementary components of $\Omega$ need not be fat, they have a \textit{fat shadow} in $\Omega$. Namely, for each bounded complementary component $p$ of $\Omega$ we can find a ball $B\subset \Omega$ with roughly the same diameter as $p$ and such that $p$ is contained in a definite enlargement of $B$. So, in some sense, $B$ is the shadow of $p$. Moreover, using the packing condition we prove that these balls have bounded overlap.

In Section \ref{section:uniformization}, which is independent from Section \ref{section:john}, we verify that Koebe's conjecture is true for domains satisfying the packing condition and the fat shadow condition. The geometry of such domains resembles the geometry of cofat domains, so Schramm's method can be applied with suitable adaptations. Specifically, we approximate $\Omega$ from outside by domains $\Omega_j$ such that $\Omega_j\supset \Omega$ and the complementary components of $\Omega_j$ are exactly the complementary components of $\Omega$ with diameter at least $1/j$. The packing condition implies that the domain $\Omega_j$ is finitely connected for each $j\in \N$. Thus, $\Omega_j$ can be mapped onto a circle domain $D_j$ with a conformal map $f_j$ by the Koebe uniformization theorem. Using suitable normalizations, and after passing to a subsequence, we may assume that $f_j$ converges to a conformal map $f$ on $\Omega$. So, the question is whether the domain $f(\Omega)$ is a circle domain. 

Schramm establishes this under the assumption that there exists $\tau>0$ such that $\Omega_j$ is $\tau$-cofat for each $j\in \N$. The main technical result in this direction is what he calls the \textit{extended Carath\'eodory kernel convergence theorem} \cite{Schramm:transboundary}*{Theorem 3.1}. We generalize this result in Theorem \ref{theorem:general_caratheodory}, disposing of the cofatness assumption and using a close-to-optimal general assumption. Then we show that our assumption holds if the transboundary modulus in $\Omega_j$ of curves joining the boundary components of an annulus $\{z\in \C:r<|z-z_0|<R\}$ tends to $0$ as $j\to\infty$ and $r\to 0$; see Lemma \ref{lemma:potential}. Finally, we prove that this condition holds under the packing condition and the fat shadow condition. We also prove that, under the limiting map $f$ from a John domain $\Omega$ onto a circle domain $D$, the point components of $\widehat \C\setminus \Omega$  correspond bijectively to the point components of $\widehat \C\setminus D$; see Theorem \ref{theorem:koebe_nondegenerate}. Theorem \ref{theorem:intro} is restated in an expanded form as Theorem \ref{theorem:john:restate}.

\section{John domains}\label{section:john}

In this section we use the Euclidean metric in $\C$ exclusively. For $z_0\in \C$ and $0<r<R$ we define $\mathbb A(z_0;r,R)=\{z\in \C:r<|z-z_0|<R\}$, $\D(z_0,r)=\{z\in \C : |z-z_0|<r\}$, and  $\mathbb S(z_0,r)=\partial \D(z_0,r)$. For $k>0$ and a ball $B=\D(z_0,r)$ we set $kB=\D(z_0,kr)$. A Jordan curve (resp.\ Jordan arc) in a metric space $X$ is a continuous and injective embedding $\gamma$ of the unit circle (resp.\ closed  interval) into $X$. For a curve $\gamma\colon [a,b]\to X$ we denote by $|\gamma|$ its trace, i.e., the image $\gamma([a,b])$. We often make no distinction between a curve and its trace. A \textit{continuum} in $X$ is a compact and connected set.  If $G\subset \widehat \C$ is a domain and $\gamma\colon [a,b]\to \bar G$ is continuous and injective map with $\gamma(a),\gamma(b)\in \partial G$ and $\gamma((a,b))\subset G$, then  $C=\gamma((a,b))$ is called a \textit{crosscut} of $G$.

\begin{lemma}\label{lemma:crosscut_connected}
If $G\subset \widehat \C$ is a simply connected domain and $C$ is a crosscut of $G$, the following statements are true.
\begin{enumerate}[label=\normalfont(\arabic*)]
\item\label{lemma:crosscut_connected:1} The set $G\setminus C$ has exactly two components $G_1,G_2$ such that
$$G\cap \partial G_1=G\cap \partial G_2=C.$$
\item\label{lemma:crosscut_connected:2} The set $\partial G_i\setminus C$ is connected for $i=1,2$.
\end{enumerate}

\end{lemma}

\begin{proof}
By the Riemann mapping theorem there is a conformal map $\phi$ from $\D$ onto $G$. By \cite{Pommerenke:conformal}*{Prop.\ 2.14}, $\phi^{-1}(C)$ is an open Jordan arc $J\subset \D$ with distinct endpoints on $\partial \D$. By the Jordan curve theorem, $\D\setminus J$ is the union of two Jordan regions $\Omega_1,\Omega_2$. We set $G_i=\phi(\Omega_i)$ for $i=1,2$ and \ref{lemma:crosscut_connected:1} follows. Fix $i\in \{1,2\}$. The set $\partial G_i\setminus C$ is precisely the set of accumulation points of $\phi(z)$ as $z\to E=\partial \Omega_i\setminus J$. There exists a decreasing sequence of Jordan regions $U_n\subset \Omega_i$, $n\in \N$, such that $\partial U_n$ is the union of $E$ with an arc in $\Omega_i$ that has the same endpoints as $E$ for each $n\in \N$. We have
$$\partial G_i\setminus C= \bigcap_{n=1}^\infty \bar{\phi(U_n)},$$
so $\partial G_i\setminus C$ is the intersection of a decreasing sequence of continua. Thus, it is a continuum and \ref{lemma:crosscut_connected:2} follows. 
\end{proof}

A metric space $(X,d)$ has \textit{bounded turning} if there exists a constant $L\geq 1$ such that for each pair $x,y\in X$ there exists a connected set $E\subset X$ that contains $x$ and $y$ with $\diam E\leq L d(x,y)$. It is a consequence of \cite{Whyburn:topology}*{II.5.2} that if, in addition, $(X,d)$ is locally compact, then for every $L'>L$ any two points $x,y\in X$ can be connected by a path $\gamma$ with $\diam |\gamma|\leq L'd(x,y)$. In particular, $(X,d)$ is pathwise connected, and thus \textit{arcwise connected} \cite{Wilder:topology}*{Corollary 31.6}; that is, any two points of $X$ can be connected by a Jordan arc.

\begin{theorem}[Bounded turning]\label{theorem:bounded_turning}
Let $\Omega\subset \C$ be an $L$-John domain for some $L\geq 1$. Then each component of $\C\setminus \Omega$ has $(L+1)$-bounded turning and in particular, it is arcwise connected.
\end{theorem}

\begin{proof}
Let $P$ be a component of $\C\setminus \Omega$. First, we consider $w_1,w_2\in  P$ such that the crosscut $C=(w_1,w_2)$ is disjoint from $P$. Since $P$ is a component of $\C\setminus \Omega$, we conclude that $\C\setminus P$ is connected. Also, if $P$ is bounded, then it is compact, and $\widehat \C\setminus  \bar P$ is a simply connected domain. If $P$ is unbounded, then $\bar P=P\cup\{\infty\}$. The domain $\widehat \C\setminus \bar P =\C\setminus P$ is again simply connected. In both cases, by Lemma \ref{lemma:crosscut_connected} the crosscut $C=(w_1,w_2)$ separates $\widehat \C\setminus \bar P$ into two components $U,V$ such that $\partial U\setminus \bar P=\partial V\setminus \bar P=C$. In particular, $w_1,w_2$ lie in both $\partial U$ and $\partial V$.

Let $w_0=(w_1+w_2)/2$ and $r=|w_1-w_2|/2$. Suppose that $U,V$ are not contained in $\D(w_0,(L+1)r)$. Since $C$ intersects the boundaries of both $U,V$ and is contained in $\D(w_0,r)$, we conclude that both $U,V$ intersect $\D(w_0,(L+1)r)$. By connectedness, since $U,V$ are not contained in $\D(w_0,(L+1)r)$, both must intersect the circle $\mathbb S(w_0,(L+1)r)$. Since $\Omega$ is a John domain, there exists a curve $\gamma\colon [0,1]\to \Omega$ with $\gamma(0)\in \mathbb S(w_0,(L+1)r)\cap U$, $\gamma(1)\in \mathbb S(w_0,(L+1)r)\cap V$, and  
$$\min\{\ell(\gamma|_{[0,t]}),\ell(\gamma|_{[t,1]})\} \leq L\dist(\gamma(t),\partial \Omega)\,\,\, \text{for all $t\in [0,1]$}.$$
Since $C=\partial U\setminus \bar P=\partial V\setminus \bar P$,  there exists $t_0\in (0,1)$ such that $\gamma(t_0)\in C\subset \D(w_0,r)$. We have $\ell(\gamma|_{[0,t_0]})> Lr$, $\ell(\gamma|_{[t_0,1]})> Lr$, and $\dist(\gamma(t_0),\partial \Omega)\leq \ell(C)/2=r$. Thus,
$$Lr< \min\{\ell(\gamma|_{[0,t_0]}),\ell(\gamma|_{[t_0,1]})\}\leq L\dist(\gamma(t_0),\partial \Omega)\leq  Lr,$$
a contradiction. Therefore, one of the sets $U,V$ is contained in $\D(w_0,(L+1)r)$. Without loss of generality, we suppose that $V\subset \D(w_0,(L+1)r)$. By Lemma \ref{lemma:crosscut_connected} \ref{lemma:crosscut_connected:2}, the set $E=\partial V\setminus C$ is connected. Also, $E\subset P\cap \bar \D(w_0,(L+1)r)$, it contains $w_1,w_2$, and $\diam E\leq 2(L+1)r =(L+1)|w_1-w_2|$.

Now, suppose that $w_1,w_2\in P$ are arbitrary distinct points. If $[w_1,w_2]\subset P$ there is nothing to show, so we suppose that $[w_1,w_2]\setminus P\neq \emptyset$. The set $[w_1,w_2]\setminus P$ has countably many components $\{C_i\}_{i\in I}$ each of which is a segment in $\C\setminus P$ with endpoints in $\partial P$. By the above, for each $i\in I$ there exists a continuum $E_i\subset P$ that contains the endpoints of $C_i$ and satisfies $\diam E_i\leq (L+1)\diam C_i$. By replacing each subsegment $C_i$ of $[w_1,w_2]$ with $E_i$ we obtain a continuum $E\subset P$ that contains $w_1,w_2$ and satisfies $\diam E\leq (L+1)|w_1-w_2|$. This completes the proof.
\end{proof}

\begin{lemma}\label{lemma:count}
Let $\{E_i\}_{i\in I}$ be a collection of closed sets in $\C$. The following are quantitatively equivalent. 
\begin{enumerate}[label=\normalfont(\arabic*)]
	\item\label{lemma:count:1} There exists $N\in \N$ such that for each $z_0\in \C$ and $r>0$ the number of sets $E_i$, $i\in I$, that intersect both complementary components of the annulus $A=\mathbb A(z_0;r/2,r)$ is bounded above by $N$.
	\item\label{lemma:count:2} For each $a\geq 1$ there exists $N(a)\in \N$ such that for each $z_0\in \C$ and $r>0$ the number of sets $E_i$, $i\in I$, that intersect the ball $\D(z_0,r)$ and have diameter at least $r/a$ is bounded above by $N(a)$.
\end{enumerate}
\end{lemma}
\begin{proof}
	Suppose that \ref{lemma:count:2} is true and consider an annulus $A=\mathbb A(z_0;r/2,r)$. If $E_i$ intersects both complementary components of $A$, then $E_i\cap \D(z_0,r)\neq \emptyset$ and $\diam E_i\geq r/2$. By \ref{lemma:count:2}, the number of such sets is bounded above by $N(2)$. Conversely, suppose that \ref{lemma:count:1} holds. Let $a\geq1$ and consider a ball $\D(z_0,r)$. We cover $\D(z_0,r)$ by a finite number of balls $D_j=\D(z_j,r/(4a))$, $j\in \{1,\dots,m\}$, where $m\leq M(a)$. Let $j\in \{1,\dots,m\}$. If   $E_i\cap D_j\neq \emptyset$ and $\diam E_i\geq r/a$, then $E_i$ is not contained in $2D_j$. Thus, $E_i$ intersects both complementary components of $\mathbb A(z_j;r_j/(4a),r_j/(2a))$. By \ref{lemma:count:1}, we conclude that the number of such sets is at most $N$. Thus, the number of sets $E_i$ that intersect $\D(z_0,r)$ and have diameter at least $r/a$ is at most $M(a)N$.
\end{proof}

\begin{proposition}[Packing condition]\label{prop:packing} Let $\Omega\subset \C$ be an $L$-John domain for some $L\geq 1$. Then for each $z_0\in \C$ and $r>0$ the number of components of $\C\setminus \Omega$ that intersect $\D(z_0,r)$ and have diameter at least $r/a$, where $a\ge 1$, is bounded above by a constant  $N(L,a)$.
\end{proposition}

\begin{proof}
Consider an annulus $\mathbb A(z_0;r,2r)$. By Lemma \ref{lemma:count} it suffices to show that the number of components of $\C\setminus \Omega$ that intersect both complementary components of the annulus is bounded above by a constant $N(L)$. Let $E_1,\dots,E_n$ be distinct components of $\C\setminus \Omega$ that intersect both complementary components of $\mathbb A(z_0;r,2r)$. If $\D(z_0,r)\subset \C\setminus \Omega$, then it is immediate that $n=1$, so we assume that $\D(z_0,r)\cap \Omega\neq \emptyset$ and let $w_0\in \D(z_0,r)\cap \Omega$. Also, we suppose that $n\geq 2$.

Each $E_i$ intersects both boundary components of the annulus $A=\mathbb A(z_0;4r/3,2r )$. We first show that $A\cap \Omega$ has at least $n$ components intersecting the circle $C=\mathbb S(z_0,5r/3)$. By Theorem \ref{theorem:bounded_turning} for each $i\in \{1,\dots,n\}$ there exists a Jordan arc $\gamma_i$ in $E_i\cap \bar A$ with endpoints in distinct components of $\partial A$ but otherwise contained in $A$. By the Jordan curve theorem $A\setminus \bigcup_{i=1}^n|\gamma_i|$ is the union of exactly $n$ Jordan regions $U_1,\dots,U_n$. The numbering is such that $U_i$ is bounded by $|\gamma_i|$, $|\gamma_{i+1}|$, and two arcs from different components of $\partial A$; here $\gamma_{n+1}\equiv \gamma_1$.  

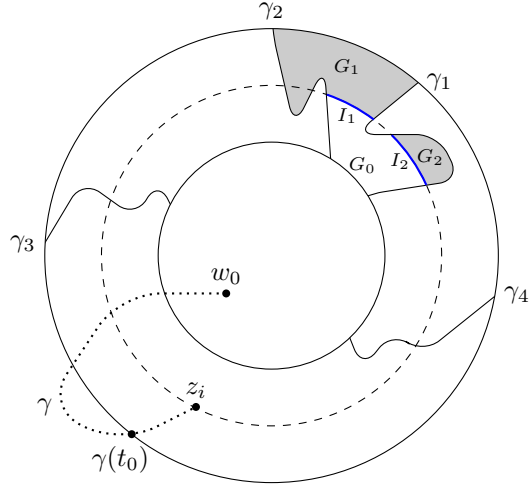
\begin{figure}
	\centering
	\begin{tikzpicture}
		\begin{scope}		
		\clip (0,0) circle (3cm);
		\fill[black!20, rounded corners=8pt] (0,4)--(0,3)--(0.3,1.7)--(0.7,2.5)--(0.8,1)--(0,0)--(1.2,0.8)--(2.5,1)--(2.2,1.6)--(1.1,1.6)--(2.2,2.5)--(3,3)--cycle;
		
		\fill[white](0,0) -- (0:2.25cm)  arc[start angle=0, end angle=75, radius=2.25cm] -- cycle;
		\draw[dashed](0,0) circle (2.25cm);
		\draw[rounded corners=8pt] (0,4)--(0,3)--(0.3,1.7)--(0.7,2.5)--(0.8,1)--(0,0)--(1.2,0.8)--(2.5,1)--(2.2,1.6)--(1.1,1.6)--(2.2,2.5)--(3,3)--cycle;
		\node at (1.2,1.2) {$\scriptstyle G_0$};
		\node at (1,2.5) {$\scriptstyle G_1$};
		\node at (2.1,1.3) {$\scriptstyle G_2$};
		\draw[thick,blue] (24.5:2.25cm) arc[start angle=24.5, end angle=45, radius=2.25cm];
		\draw[thick,blue] (53:2.25cm) arc[start angle=53, end angle=71, radius=2.25cm];
		\node at (1.7,1.25) {$\scriptstyle I_2$};	
		\node at (1,1.85) {$\scriptstyle I_1$};	
		\draw[rounded corners=8pt] (-1,0)--(-1.5,1)--(-1.8,0.5)--(-2.5,1)--(-3.1,0);
		\draw[rounded corners=8pt] (1,-1)--(1.2,-1.5)--(1.6,-1)--(2,-1.3)--(3,-0.5);
		\draw[fill=white] (0,0) circle (1.5cm);
		\end{scope}
		
		\draw (0,0) circle (3cm);
		\node[right] at (50:3cm) {$\gamma_1$};
		\node[above] at (90:3cm) {$\gamma_2$};
		\node[left] at (177:3cm) {$\gamma_3$};
		\node[right] at (350:3cm) {$\gamma_4$};
		\fill (-0.6,-0.5) circle (1.5pt) node[above]{$w_0$};
		\draw[dotted, thick, rounded corners=15pt] (-0.6,-0.5)--(-2,-0.5)--(-3,-2)--(-2,-2.5)--(-1,-2);
		\node at (-3,-2) {$\gamma$};
		\fill (-1,-2) circle (1.5pt)node[above]{$z_i$};
		\fill (-1.85,-2.36) circle (1.5pt) node[below,yshift=-0.03cm, xshift=-0.15cm]{$\gamma(t_0)$};
	\end{tikzpicture}
	\caption{The Jordan arcs $\{\gamma_i\}_{i=1}^n$ split the annulus $A$ into $n$ Jordan regions. Shown is one such region $G$ and the crosscut $I_1\subset C$ that is the common boundary of domains $G_0,G_1\subset G$, whose boundaries contain arcs of $\partial A$.}\label{fig:packing}
\end{figure}

We fix $i\in \{1,\dots,n\}$ and we show that $U_i\cap \Omega$ has a component $D_i$ that intersects $C$. Let $G=U_i$ and let $G_0$ be the component of $U_i\setminus C$ that contains in its boundary an arc of the inner circle of $\partial A$. By \cite{Pommerenke:conformal}*{Prop.\ 2.13} there exist countably many crosscuts $I_k\subset C$ of $G$ such that $I_k=G\cap \partial G_k\subset G\cap \partial G_0$, where $k\geq 1$, for a disjoint collection of domains $\{G_k\}_{k\geq 0}$ that satisfy
$$G=G_0\cup \bigcup_{k\geq 1}G_k\cup \bigcup_{k\geq 1}I_k;$$
see Figure \ref{fig:packing}. The arc of the outer circle of $\partial A$ that is part of the boundary of $G=U_i$ lies in the boundary of precisely one of the sets $\{G_k\}_{k\geq 1}$, say, in $G_1$. If both endpoints of the crosscut $I_1$ lie in $|\gamma_i|$, then together with an arc of $|\gamma_i|$ they bound a Jordan region, the boundary of which does not intersect $\partial A$, a contradiction. Similarly, not both endpoints of $I_1$ lie in $|\gamma_{i+1}|$. Therefore, the crosscut $I_1$ has one endpoint in $|\gamma_i|$ and one endpoint in $|\gamma_{i+1}|$. We set $C_i=I_1$. Next, suppose that $U_i\cap C\cap \Omega= \emptyset$ for some $i\in \{1,\dots,n\}$. Since $C_i\cap \Omega=\emptyset$ we have $C_i\subset \C\setminus \Omega$ and hence $\bar {C_i}\subset \C\setminus \Omega$ and $\bar{C_i}$ is contained in a component of $\C\setminus \Omega$. However, $\bar {C_i}$ intersects both components $E_i$ and $E_{i+1}$, which is a contradiction. Therefore, for each $i\in \{1,\dots,n\}$ the set $U_i\cap \Omega$ has a component $D_i$ that intersects $C$. Finally, note that $D_i$ is also a component of $A\cap \Omega=\bigcup_{i=1}^n(U_i\cap \Omega)$.

We fix $i\in \{1,\dots,n\}$ and let $z_i\in D_i\cap C$. Since $\Omega$ is an $L$-John domain, there exists a curve $\gamma\colon[0,1]\to \Omega$ such that $\gamma(0)=w_0$, $\gamma(1)=z_i$, and
\[\min\{\ell(\gamma|_{[0,t]}),\ell(\gamma|_{[t,1]})\}\le L\dist (\gamma(t),\partial \Omega) \,\,\, \text{for all $t\in [0,1]$.}\]
Let $t_0=\max \{t\in [0,1]:\gamma(t)\cap \partial A\neq \emptyset\}$; see Figure \ref{fig:packing}. We have $\gamma(t_0)\in \partial A$, so $|\gamma(t_0)-w_0|\geq r/3$ and $|\gamma(t_0)-z_i|\geq r/3$. Note that  $\gamma((t_0,1])\subset D_i$, so $\gamma(t_0)\in \partial D_i\cap \partial A\cap \Omega$. Moreover, since $D_i$ is a component of $A\cap \Omega$, the component of $\partial D_i\cap \partial A$ that contains $\gamma(t_0)$ is an arc $S_i $ that has length 
\begin{align*}
\ell(S_i)\geq \dist(\gamma(t_0),\partial \Omega) \geq L^{-1}\min\{|\gamma(t_0)-w_0|,|\gamma(t_0)-z_i|\}\geq  \frac{r}{3L}.
\end{align*}
Note that if $i\neq j$, then the arcs $S_i,S_j\subset \partial A$ are non-overlapping since $D_i,D_j$ are components of $A\cap \Omega$.  Thus, 
$$n\cdot \frac{r}{3L}\leq \sum_{i=1}^n \ell(S_i) \leq \ell(\partial A)= 2\pi (2r+ 4r/3).$$
We conclude that $n\leq 20\pi L\eqqcolon N(L)$. Our claim is proved.
\end{proof}

\begin{lemma} \label{diamanddistbounded} Let $\Omega\subset \C$ be a bounded $L$-John domain for some $L\geq 1$. Then there exists a point $z_0\in \Omega$ such that
\[ \frac{1}{4L}\diam \Omega \le \dist (z_0,\partial \Omega)\le \diam \Omega.\]
\end{lemma}

\begin{proof} Let $d=\diam \Omega$. Then there are two points $\zeta_1,\zeta_2\in \partial \Omega$ such that $|\zeta_1-\zeta_2|=d$. We take $z_1,z_2 \in \Omega$ such that $|z_i-\zeta_i|\le d/4$ for $i\in\{1,2\}$. Thus, we have
\[|z_1-z_2|\ge d-d/4-d/4=d/2.\]  
Since $\Omega$ is an $L$-John domain, there exists a curve $\gamma \colon [0,1]\to \Omega$ such that $\gamma(0)=z_1$, $\gamma(1)=z_2$, and
\begin{align}\label{johncurvepropert}
\min\{\ell (\gamma|_{[0,t]}), \ell(\gamma|_{[t,1]} )\} \leq L \dist(\gamma(t),\partial \Omega)\,\,\, \text{for all $t\in [0,1]$.}
\end{align}
Since $|z_1-z_2|\geq d/2$, the curve $\gamma$ is not contained in $\D(z_1,d/4)$. Thus, there exists $t_0\in (0,1)$ such that $|z_1-\gamma(t_0)|=d/4$. We have
\[|\gamma(t_0)-z_2|=|(z_1-z_2)-(z_1-\gamma (t_0))|\ge d/2-d/4=d/4.\] 
Therefore, if $z_0=\gamma(t_0)$ then by \eqref{johncurvepropert} we infer that
\[ d\ge \dist (z_0,\partial \Omega)\ge L^{-1}\min\{|z_0-z_1|,|z_0-z_2|\}=\frac{d}{4L},\]
which completes the proof.
\end{proof}

Recall that if $\Omega\subset \widehat \C$ is a domain, then $\mathcal C(\Omega)$ is the set of components of $\widehat \C\setminus \Omega$. Note that bounded components of $\C\setminus \Omega$ correspond precisely to the elements of $\mathcal C(\Omega)$ that do not contain $\infty$. 

\begin{proposition}[Bounded overlap]\label{prop:overlap}
Let $\Omega\subset \widehat \C$ be a domain satisfying the following conditions for some $L\geq 1$.
\begin{enumerate}[label=\normalfont(\arabic*)]

\item\label{prop:overlap:1} For each $z_0\in \C$ and $r>0$ the number of components $p\in \mathcal C(\Omega)$ that intersect both complementary components of the annulus $\mathbb A(z_0;r/2,r)$ is bounded above by $L$.
\item\label{prop:overlap:2} There exists a set $\mathcal C\subset \mathcal C(\Omega)$ such that for each $p\in \mathcal C$ there exists a ball $B_p\subset \Omega$ with
\begin{align*}
\diam B_p \leq L\diam p\quad \text{and}\quad  p\subset LB_p.
\end{align*}
\end{enumerate}
Then $$\sum_{p\in \mathcal C}\chi_{\frac{1}{2}B_p}\leq N(L).$$
\end{proposition}

\begin{proof} 
For $p\in \mathcal C$ we write $B_p=\D(z_p,r_p)$. Suppose that $p\in \mathcal C$ is fixed and $q\in \mathcal C$ is such that $\frac{1}{2}B_p\cap \frac{1}{2}B_q\neq \emptyset$, so $|z_p-z_q|<(r_p+r_q)/2$. Thus, for each $z\in LB_q$ we have
\begin{align*}
|z-z_p|\leq |z-z_q|+|z_q-z_p|<Lr_q+ (r_p+r_q)/2= (L+1/2)r_q+r_p/2,
\end{align*}
which implies that 
\begin{align}\label{prop:overlap:inclusion}
LB_q\subset \D(z_p,(L+1/2)r_q+r_p/2).
\end{align}
If $(L+1/2)r_q<r_p/2$, then by \eqref{prop:overlap:inclusion} we have $LB_q\subset B_p$. By assumption \ref{prop:overlap:2}, we obtain $q\subset B_p$, which is a contradiction, since $B_p\subset \Omega$. Therefore, we have
\begin{align*}
(L+1/2)r_q\geq r_p/2.
\end{align*}
By \eqref{prop:overlap:inclusion}, we obtain $LB_q\subset \D(z_p,(2L+1)r_q)$. By assumption \ref{prop:overlap:2}, we have $q\subset \D(z_p, (2L+1)r_q)$ and $\diam q\geq L^{-1}\diam B_q=2L^{-1}r_q$. By assumption \ref{prop:overlap:1} and Lemma \ref{lemma:count} \ref{lemma:count:2}, the number of elements $q\in \mathcal C$ that satisfy $q\cap \D(z_p, (2L+1)r_q)\neq \emptyset$ and $\diam q\geq 2L^{-1}r_q$ is bounded above by a constant $N(L)$. Therefore, the number of elements $q\in \mathcal C$ such that $\frac{1}{2}B_p\cap \frac{1}{2}B_q\neq \emptyset$ is bounded above by $N(L)$.
\end{proof}

\begin{proposition}[Fat shadows]\label{prop:shadows} 
Let $\Omega\subset \C$ be an $L$-John domain for some $L\geq 1$. Let $\mathcal C$ be the set of non-degenerate components $p\in \mathcal C(\Omega)$ that do not contain $\infty$.  Then there is a constant $k=k(L)\ge 1$ such that for each $p\in \mathcal C$ there exists a ball $B_p\subset \Omega$ with
\begin{align}\label{prop:shadow:conclusion}
\diam B_p \le k\diam p, \quad p\subset kB_p, \quad \text{and} \quad \sum_{p\in \mathcal C}\chi_{B_p}\leq k.
\end{align}
\end{proposition}

\begin{proof} 
Let $p\in \mathcal C$ and $d_p=\diam p$. Suppose first that $\Omega$ is bounded. By Lemma \ref{diamanddistbounded} there is a point $z_0\in \Omega$ such that for $r_0=\dist(z_0,\partial \Omega)$ we have
\begin{align}\label{r0d}
\frac{1}{4L}\diam \Omega\le r_0 \le \diam\Omega.
\end{align}
Let $x_p\in \partial p$ and $z_p\in \Omega$ such that $|x_p-z_p|< \frac{d_p}{8L}$. There is a curve $\gamma_p$ in $\Omega$ connecting $z_0$ to $z_p$ as in Definition \ref{definition:john}. We observe that $|x_p-z_0|\ge r_0$ and 
\begin{align}\label{distdiam}
\diam |\gamma_p|&\ge |z_p-z_0| \ge \big||z_p-x_p|-|x_p-z_0|\big|> r_0-\frac{d_p}{8L}\geq \frac{d_p}{8L},
\end{align}
where the last inequality follows from \eqref{r0d} and the fact that $d_p\leq \diam \Omega$.

Now, suppose that $\Omega$ is unbounded and our goal is to show a similar inequality as \eqref{distdiam}.  There is $R_p>0$ such that $p\subset \D(0,R_p)$ and hence $d_p=\diam p \in (0,2R_p)$. Since $\Omega$ is an unbounded domain, there is a point $z_0\in \Omega\cap \mathbb{S}(0,6R_p)$. Let $x_p\in\partial p$ and $z_p\in \Omega$ such that $|z_p-x_p|< \frac{d_p}{8L} < \frac{d_p}{2}< R_p$. Then there exists a curve $\gamma_p$ in $\Omega$ connecting $z_0$ to $z_p$ as in Definition \ref{definition:john}. Note that $z_p\in \D(0,2R_p)$, so  
\begin{align}\label{distdiam:unbounded}
\diam |\gamma_p|\ge |z_p-z_0|> 4R_p> 2d_p > \frac{d_p}{8L}.
\end{align}
That is, inequality \eqref{distdiam} is also true in this case.

Hence, in both cases there is a point $w_p\in |\gamma_p|$ with $|w_p-z_p|=\frac{d_p}{16L}$. So, we have
\begin{align}\label{assumdistance}
|w_p-x_p| \le |w_p-z_p|+|z_p-x_p|< \frac{d_p}{16L}+ \frac{d_p}{8L}=\frac{3d_p}{16L}.
\end{align}
Moreover, as we proved in \eqref{distdiam} and \eqref{distdiam:unbounded}, we have
\[|z_p-z_0|> \frac{d_p}{8L} \]
and thus
\begin{align}\label{forjpro}
|w_p-z_0|\ge \big||w_p-z_p|-|z_p-z_0|\big|> \frac{d_p}{8L}- \frac{d_p}{16L}=\frac{d_p}{16L}.
\end{align}

Let $B_p=\D(w_p,r_p)$, where $r_p=\dist (w_p,\partial \Omega)$, so $B_p\subset \Omega$. By \eqref{assumdistance} we infer that
\[ r_p\le |w_p-x_p|< \frac{3d_p}{16L}. \]
Hence,
\begin{align}\label{ann1}
\diam B_p< \frac{3}{8L}\diam p<\diam p.
\end{align}
Moreover, by \eqref{forjpro} and the John property we deduce that
\begin{align}\label{assumjohn}
Lr_p&=L\dist(w_p,\partial \Omega)\geq \min \{|w_p-z_0|,|w_p-z_p|\}= \frac{d_p}{16L}.
\end{align}
Now, let $z\in p$. By \eqref{assumdistance} and \eqref{assumjohn} we have
\[|z-w_p|\le |z-x_p|+|x_p-w_p|< d_p+\frac{3d_p}{16L}< \frac{3d_p}{2}\le 24L^2r_p.\]
Hence, $z\in \D(w_p,24L^2r_p)=24L^2B_p$.  So, we have proved that $p\subset 24L^2B_p$. 

We now apply Proposition \ref{prop:overlap} in order to obtain the final inequality in \eqref{prop:shadow:conclusion} as follows. Condition \ref{prop:overlap:1} in Proposition \ref{prop:overlap} is a consequence of Proposition \ref{prop:packing} and Lemma \ref{lemma:count}. We have shown that the balls $B_p$, $p\in \mathcal C$, satisfy $\diam B_p\leq \diam p$ (see \eqref{ann1}) and $p\subset 24L^2B_p$, as required in Proposition \ref{prop:overlap} \ref{prop:overlap:2}. Hence, by Proposition \ref{prop:overlap}, we have
$$\sum_{p\in \mathcal C} \chi_{\frac{1}{2}B_p}\leq N(L).$$
Finally, for $p\in \mathcal C$ we set $B_p'=\frac{1}{2}B_p$ and note that $B_p'\subset \Omega$, $\diam B_p'\leq \diam p$, and $p\subset 48L^2B_p'$. We have proved the desired \eqref{prop:shadow:conclusion} for $k=\max\{48L^2,N(L)\}$. 
\end{proof}

\begin{lemma}\label{lemma:diameter}
Let $\Omega\subset \widehat \C$ be a domain satisfying the packing condition in the conclusion of Proposition \ref{prop:packing}. Then for each $\varepsilon>0$ there exist at most finitely many components $p\in \mathcal C(\Omega)$ with spherical diameter greater than $\varepsilon$.
\end{lemma}

\begin{proof}
Suppose, for the sake of contradiction, that there exists $\varepsilon>0$ and there exists a sequence of distinct components $p_n\in \mathcal C(\Omega)$, $n\in \N$, with spherical diameters greater than $\varepsilon$. At most one of these components can contain $\infty$, thus, by discarding it, we may assume that $p_n$ is a bounded component of $\C\setminus \Omega$ for each $n\in \N$. By passing to a subsequence, we may assume that $\{p_n\}_{n\in \N}$ converges to a non-degenerate continuum $E\subset \widehat \C\setminus \Omega$. Let $z,w\in E\cap \C$ with $z\neq w$. There exist sequences $z_n,w_n\in p_n$, $n\in \N$, converging to $z,w$, respectively. Thus, for $r=|z-w|/2$ and for all sufficiently large $n\in \N$ the set $p_n$ intersects $\D(z,r)$ and has Euclidean diameter at least $r$. This violates the packing condition in the conclusion of Proposition \ref{prop:packing}.   
\end{proof}

\section{Uniformization by circle domains}\label{section:uniformization}

\subsection{Classical and transboundary modulus}\label{section:modulus}
We recall the definition of classical modulus in $\widehat \C$. We equip $\widehat \C$ with the spherical metric $\sigma$ and spherical measure $\Sigma$.
Let $\Gamma$ be a family of curves in $\widehat \C$. A Borel function $\rho\colon \widehat \C\to [0,\infty]$ is \textit{admissible} for $\Gamma$ if $\int_{\gamma}\rho\, ds\geq 1$ for every locally rectifiable path $\gamma\in \Gamma$. The conformal modulus of $\Gamma$ is defined as
$$\Mod\Gamma=\inf_{\rho}\int\rho^2\, d\Sigma,$$
where the infimum is taken over all admissible functions $\rho$ for $\Gamma$. We note that if $\Gamma$ is a family of curves in $\C$ rather than in $\widehat \C$, then we may equip $\C$ instead with the Euclidean metric and Lebesgue measure in order to define modulus. The value of $\Mod\Gamma$ is not affected by this change. 

Given a domain $\Omega \subset \widehat \C$, recall that $\mathcal{C}(\Omega)$ is the collection of connected components of $\widehat \C \setminus \Omega$. We set $\hat{\Omega}=\widehat \C/\sim$, where $z\sim w$ if $z,w\in \Omega$ and $z=w$ or if $z,w\in p$ for some $p\in \mathcal C(\Omega)$. The corresponding quotient map is $\pi_\Omega:\widehat \C \to \hat{\Omega}$. Identifying each $z \in \Omega$ and $p \in \mathcal{C}(\Omega)$ with
$\pi_\Omega(z)$ and $\pi_\Omega(p)$, respectively, we have 
$$
\hat{\Omega}=\Omega \cup \mathcal{C}(\Omega). 
$$ 
By Moore's theorem \cite{Moore:theorem}, the quotient space $\hat \Omega$ is homeomorphic to $\widehat \C$. The quotient map $\pi_{\Omega}$ is \textit{monotone}, in the sense that the preimage of each point is a continuum. This implies the stronger property that the preimage of each continuum is a continuum. 

%By a theorem of Youngs \cite{Youngs:monotone}, $\pi_\Omega$ is the uniform limit of a sequence of homeomorphisms between the spheres $\widehat \C$ and $\hat \Omega$.

To simplify notation, we treat elements $p\in \mathcal C(\Omega)$ both as subsets of $\widehat \C$ and as points of $\hat \Omega$. A homeomorphism $f\colon \Omega_1 \to \Omega_2$ between domains in $\widehat \C$ has a homeomorphic extension $\hat{f}\colon \hat{\Omega}_1 \to \hat{\Omega}_2$. Namely, if $p\in \mathcal C(\Omega_1)$, then $\hat f(p)\in \mathcal C(\Omega_2)$ and for each sequence $\{z_n\}_{n\in \N}$ in $\Omega_1$ accumulating at $p$ the sequence $\{f(z_n)\}_{n\in \N}$ accumulates at $\hat f(p)$. See \cite{NtalampekosYounsi:rigidity}*{Section 3} for a detailed discussion.

We recall the definition of transboundary modulus, which was introduced by Schramm \cite{Schramm:transboundary}. Let $\Omega\subset \widehat \C$ be a domain. Let $\rho\colon \hat \Omega \to [0,\infty]$ be a Borel function and $\gamma\colon [a,b]\to  \hat \Omega$ be a curve. Then $\gamma^{-1}( \pi_{\Omega}(\Omega))$ has countably many components $O_j\subset [a,b]$, $j\in J$. For $j\in J$ define $\gamma_j= \gamma|_{O_j}$ and $\alpha_j = \pi_{\Omega}^{-1}\circ \gamma_j$. We define
\begin{align*}
\int_{\gamma} \rho \, ds= \sum_{j\in J}\int_{\alpha_j} \rho\circ \pi_{\Omega} \, ds,
\end{align*}
where the integral is understood to be infinite if one of the curves $\alpha_j$ is not locally rectifiable. Let $\Gamma$ be a family of curves in $\hat \Omega $.  We say that a Borel function $\rho\colon \hat \Omega \to [0,\infty]$ is \textit{admissible} for $\Gamma$ if 
\begin{align*}
\int_{\gamma} \rho \, ds + \sum_{\substack{p\in \mathcal C(\Omega)\\ |\gamma|\cap p\neq \emptyset}} \rho( p)\geq 1
\end{align*}
for each $\gamma\in \Gamma$.  The \textit{transboundary modulus} of $\Gamma$ with respect to the domain $\Omega$ is defined to be
\begin{align*}
\Mod_{\Omega}\Gamma=  \inf_{\rho}\left\{ \int_{\Omega}\rho^2 \, d\Sigma + \sum_{p\in \mathcal C(\Omega)}\rho(p)^2\right\},
\end{align*}
where the infimum is taken over all admissible functions $\rho$. We remark that if $\Gamma$ is a family of curves in $\pi_\Omega(\C)$, i.e., not passing through $\pi_\Omega(\infty)$, then we may equip $\C$ instead with the Euclidean metric and Lebesgue measure in order to define the transboundary modulus of $\Gamma$. As in the case of classical modulus, the value of $\Mod_\Omega\Gamma$ is not affected by this change. For simplicity, if $\Gamma$ is a family of curves in $\widehat \C$ we may use the notation $\Mod_\Omega\Gamma$ in place of $\Mod_\Omega \pi_\Omega(\Gamma)$.

It was observed by Schramm that transboundary modulus is invariant under conformal maps. Specifically, if $f\colon \Omega_1\to \Omega_2$ is a conformal map between domains $\Omega_1,\Omega_2\subset \widehat \C$, then for every curve family $\Gamma$ in $\hat \Omega_1$ we have 
$\Mod_{\Omega_1}\Gamma  =\Mod_{\Omega_2} \hat f(\Gamma)$.

The next result gives a comparison between classical and transboundary modulus in cofat domains. Recall the definition of a cofat domain from Section \ref{section:background}.

\begin{proposition}[\cite{Ntalampekos:schottky}*{Proposition 5.6}]\label{proposition:comparison}
Let $\Omega\subset \widehat \C$ be a $\tau$-cofat domain for some $\tau>0$. Then for each curve family $\Gamma$ in $\widehat \C$ we have
$$\min\{1,\Mod\Gamma\}\leq c(\tau) \Mod_{\Omega}\pi_\Omega(\Gamma).$$
\end{proposition}

\subsection{Generalization of Carath\'eodory's theorem}

A version of the Cara\-th\'e\-odory kernel convergence theorem states that if a sequence of conformal maps $f_j\colon \D\to D_j$, $j\in \N$, converges locally uniformly to a conformal map $f\colon \D\to D$, then, for any Hausdorff limit $E$ of the sequence of compact sets $\{\widehat \C\setminus D_j\}_{j\in \N}$, the domain $D$ is a component of $\widehat \C\setminus E$. See \cite{KarafylliaNtalampekos:gromov_hyperbolic}*{Lemma 3.2} for a relevant statement.

Schramm \cite{Schramm:transboundary}*{Theorem 3.1} proved an innovative generalization of this theorem for a sequence of conformal maps $f_j\colon \Omega_j\to D_j$, $j\in \N$, between multiply connected cofat domains. The following result generalizes Schramm's theorem by replacing cofatness of $\Omega_j$ with a weaker condition, namely condition \ref{theorem:general_caratheodory:2} below. See Figure \ref{fig:exhaustion} for an illustration.

\begin{theorem}\label{theorem:general_caratheodory}
Let $\Omega\subset \widehat \C$ be a domain and let $q\in \mathcal C(\Omega)$ and $\tau>0$. Suppose that the following conditions hold. 
\begin{enumerate}[label=\normalfont(\arabic*)]
\item\label{theorem:general_caratheodory:1} There exists a sequence of domains $\{\Omega_j\}_{j\in \N}$ such that each compact subset of $\Omega$ is contained in $\Omega_j$ for all sufficiently large $j\in \N$ and for each $j\in \N$ there exists $q_j\in \mathcal C(\Omega_j)$ with $q\subset q_j$.
\item\label{theorem:general_caratheodory:2} There exists a Jordan region $U\supset q$ with  $\partial U\subset \Omega$ such that 
$$\lim_{\substack{z\to  q\\z\in \Omega}} \limsup_{j\to\infty} \Mod_{\Omega_j} \mathcal G^*(z,q_j; U)=0,$$
where $\mathcal G^*(z,q_j;U)$ is the family of curves in $\pi_{\Omega_j}(\bar U)$ with endpoints on $\partial U$ that separate $z$ and $q_j$ in $\pi_{\Omega_j}(U)$. 
\item\label{theorem:general_caratheodory:3} For each $j\in \N$, there exists a conformal map $f_j$ from $\Omega_j$ onto a $\tau$-cofat domain $D_j\subset \widehat \C$ such that $\{f_j\}_{j\in \N}$ converges locally uniformly in $\Omega$ to a conformal map $f$ from $\Omega$ onto a domain $D\subset \widehat \C$.
\end{enumerate}
Then, for any Hausdorff limit $E$ of the sequence $\{\hat f_j(q_j)\}_{j\in \N}$, the set $\hat f(q)$ is the complement of the component of $\widehat \C\setminus E$ that contains $D$. 
\end{theorem}

An annulus $A\subset \widehat \C$ is an open set with exactly two complementary components. For an annulus $A\subset\widehat \C$ we denote by $\Gamma(A)$ the family of curves $\gamma\colon [a,b]\to \bar A$ with endpoints in distinct boundary components of $A$ and with $\gamma((a,b))\subset A$. We also denote by $\Gamma^*(A)$ the family of curves in $A$ separating the boundary components of $A$.

\begin{proof}
By passing to a subsequence, suppose that $\{\hat f_j(q_j)\}_{j\in \N}$ converges to a connected compact set $E$ in the Hausdorff sense. We prove that $D\cap E=\emptyset$. Let $z\in \Omega$ and $S\subset \Omega$ be a Jordan curve that separates $z$ and $q$; the existence of $S$ can be justified with the aid of Zoretti's theorem \cite{Whyburn:topology}*{Corollary VI.3.11}. By assumption \ref{theorem:general_caratheodory:1}, for all sufficiently large $j\in \N$ we have $S\subset \Omega_j$ and $q\subset q_j$, so the set $q_j$ lies in the component of $\widehat \C\setminus S$ that contains $q$. Thus, the Jordan curve $f_j(S)$ separates $\hat f_j(q_j)$ and $f_j(z)$. These three sequences of sets converge in the Hausdorff sense as $j\to\infty$ to $f(S)$, $E$, and $f(z)$, respectively. Therefore, the point $f(z)$ lies in a component of $\widehat \C\setminus f(S)$ that is disjoint from $E$. Since $z\in \Omega$ is arbitrary, we conclude that $D\cap E=\emptyset$. We denote by $V$ the component of $\widehat \C\setminus E$ that contains $D$. See Figure \ref{fig:general_caratheodory:V} for an illustration.

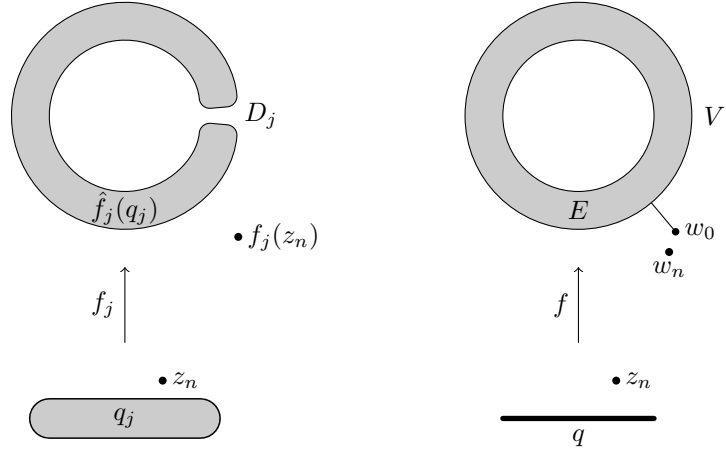
\begin{figure}
	\centering
	\begin{tikzpicture}
		\draw[fill=black!20, rounded corners] (5:1) arc (5:355:1) -- (355:1.5) arc (355:5:1.5)--cycle; 
		\node at (0,-1.22) {$\hat f_j(q_j)$};
		\node at (1.8,0) {$D_j$};
		\draw[fill=black!20] (6,0) circle (1.5cm);
		\draw[fill=white] (6,0) circle (1cm);
		\node at (6,-1.25) {$E$};
		\node at (7.8,0) {$V$};
		\draw[shift={(6,0)}] (-50:1.5)--(-50:2)node[circle, fill=black, inner sep=1pt] {} node[right] {$w_0$};
		\fill (7.2,-1.8) circle (1.5pt) node[below] {$w_n$};
		\fill (1.5,-1.6) circle (1.5pt) node[right] {$f_j(z_n)$};
		
		\begin{scope}[shift={(0,-4)}]
		\draw[line width=0.54cm, line cap=round](-1,0)--(1,0);
		\draw[line width=0.5cm, line cap=round, color=black!20](-1,0)--(1,0);
		\node at (0,0) {$q_j$};
		
		\draw[line width=2pt, line cap=round] (5,0)--(7,0)  node[below,pos=.5]{$q$};
		\fill (6.5,0.5) circle (1.5pt) node[right]{$z_n$};
		\fill (0.5,0.5) circle (1.5pt) node[right]{$z_n$};
		\end{scope}
		
		\draw[->] (0,-3)--(0,-2) node[left,pos=.5]{$f_j$};
		\draw[->] (6,-3)--(6,-2) node[left,pos=.5]{$f$};
	\end{tikzpicture}
	\caption{The set $E$ that is the Hausdorff limit of $\{\hat f(q_j)\}_{j\in \N}$ is a closed annulus in this example. The set $V$ is the complementary component of this annulus that contains $\infty$ and $Q=\widehat \C\setminus V$ is a closed disk. Finally, $\hat f(q)$ is the union $Q$ with a line segment in $V$.}\label{fig:general_caratheodory:V}
\end{figure}

Let $p\in \mathcal C(\Omega)\setminus \{q\}$. We show that $\hat f(p)\subset V$. Consider a Jordan curve $S\subset \Omega$ that separates $p$ and $q$. Let $z\in \Omega$ be a point in the component of $\widehat \C\setminus S$ that contains $p$. By the above argument, $f(z)$ lies in a component of $\widehat \C\setminus f(S)$ that is disjoint from $E$. By the choice of $z$, the set $\hat f(p)$ lies in the same component of $\widehat \C\setminus f(S)$ as $f(z)$, so $\hat f(p)\cap E=\emptyset$. Thus, $\hat f(p)$ lies in a component of $\widehat \C\setminus E$.  Since $\partial \hat f(p)\subset \bar D$, we see that $\hat f(p)\subset V$, as claimed. Summarizing, we have proved that
$$\widehat \C\setminus \hat f(q)=D\cup \bigcup_{p\in \mathcal C(\Omega)\setminus \{q\}} \hat f(p) \subset V.$$
Equivalently, if we set $Q=\widehat \C\setminus V$, then $Q\subset \hat f(q)$. Our goal is to show the reverse inclusion, namely $\hat f(q)\subset Q$.  

We suppose, instead, that $\hat f(q)\setminus Q\neq \emptyset$, so $\hat f(q)\cap V\neq \emptyset$, as in Figure \ref{fig:general_caratheodory:V}. Since $V$ is a connected open set, there exists a path in $V$ connecting a point of $\hat f(q)$ to a point of $D$. Thus, $\partial \hat f(q)\cap V\neq \emptyset$, so there exists a point $w_0\in \partial \hat f(q)\setminus Q$. Consider a sequence $w_n\in D$, $n\in \N$, converging to $w_0$. Let $z_n=f^{-1}(w_n)\in \Omega$, $n\in \N$, and note that $z_n$ accumulates at $q$ as $n\to\infty$. For fixed $n\in \N$ we have $f_j(z_n)\to w_n$ as $j\to\infty$.  

We claim that there exist $C_0>0$ and $N\in \N$ such that for each $n\geq N$ we have
\begin{align}\label{theorem:general_caratheodory:loewner}
\Mod_{D_j}\hat f_j(\mathcal G^*(z_n, q_j;U))\geq C_0\,\,\, \text{for all sufficiently large $j\in \N$.} 
\end{align}
Assuming this, by the conformal invariance of transboundary modulus, we have
$$\Mod_{\Omega_j}\mathcal G^*(z_n,q_j;U)\geq C_0$$
for each $n\geq N$ and for all sufficiently large $j\in \N$. Taking the $\limsup$ as $j\to\infty$ and then the limit as $n\to\infty$ we obtain a contradiction to assumption \ref{theorem:general_caratheodory:2}.

We now prove \eqref{theorem:general_caratheodory:loewner}. Consider two disjoint Jordan curves $J_1,J_2$ in the simply connected domain $V$ that separate $w_0$ from $E$ and intersect the component of $\widehat \C\setminus f(\partial U)$ that is disjoint from $E$. Since $w_n\to w_0$ as $n\to\infty$, we see that for all sufficiently large $n\in \N$, say for $n\geq N$, the curves $J_1,J_2$ separate $w_n$ from $E$. We fix $n\geq N$. Since $f_j(z_n)\to w_n$, $\hat f_j(q_j)\to E$, and $f_j(\partial U)\to f(\partial U)$ as $j\to\infty$, we conclude that for all sufficiently large $j\in \N$, say for $j\geq j(n)$, the curves $J_1,J_2$ separate $f_j(z_n)$ from $\hat f_j(q_j)$ and intersect the component of $\widehat \C\setminus f_j(\partial U)$ that is disjoint from $\hat f_j(q_j)$. See Figure \ref{fig:annulus_separating} for an illustration.

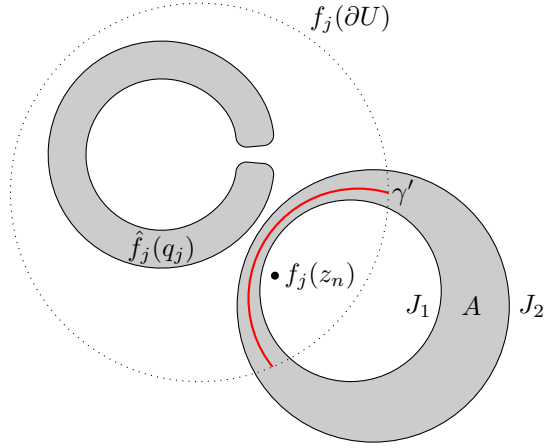
\begin{figure}
	\centering
	\begin{tikzpicture}
		\draw[fill=black!20, rounded corners] (5:1) arc (5:355:1) -- (355:1.5) arc (355:5:1.5)--cycle; 
		\node at (0,-1.22) {$\hat f_j(q_j)$};
		\draw[fill=black!20] (2.8,-2) circle (1.8cm);
		\draw[fill=white] (2.5,-1.8) circle (1.2cm);
		\draw[dotted] (0.5,-0.5) circle (2.5cm);
		\fill (1.5,-1.6) circle (1.5pt) node[right] {$f_j(z_n)$};
		\node at (3.4,-2) {$J_1$};
		\node at (4.9,-2) {$J_2$};
		\node at (4.1,-2) {$A$};
		\node at (2.5,1.8) {$f_j(\partial U)$};
		\begin{scope}
		\clip (0.5,-0.5) circle (2.5cm);
		\draw[thick,red] (2.6,-1.9) circle (1.45cm);
		\end{scope}
		\node at (3.2,-0.5) {$\gamma'$};
	\end{tikzpicture}
	\caption{The curves $J_1,J_2$ separate $f_j(z_n)$ from $\hat f_j(q_j)$ and intersect the exterior of $f_j(\partial U)$. Shown is a curve $\gamma'\in \Gamma(j,n)$.} \label{fig:annulus_separating}
\end{figure}

Consider the annulus $A$ bounded by $J_1$ and $J_2$.  Each $\gamma\in \Gamma^*(A)$ separates $f_j(z_n)$ from $\hat f_j(q_j)$ and intersects the component of $\widehat \C\setminus f_j(\partial U)$ that is disjoint from $\hat f_j(q_j)$. Thus, $\gamma$ has a subcurve $\gamma'$ with endpoints on $f_j(\partial U)$ that separates $f_j(z_n)$ from $\hat f_j(q_j)$ in the Jordan region bounded by $f_j(\partial U)$ and which contains $\hat f_j(q_j)$. Let $\Gamma(j,n)$ be the family of these subcurves. Then 
$$\Mod\Gamma(j,n)\geq \Mod \Gamma^*(A)>0.$$
By assumption, the domain $D_j$ is $\tau$-cofat for each $j\in \N$. By the comparison between classical and transboundary modulus from Proposition \ref{proposition:comparison}, we have
$$\Mod_{D_j} \pi_{D_j}(\Gamma(j,n)) \geq c(\tau)^{-1} \min\{1, \Mod\Gamma^*(A)\}\eqqcolon C_0$$
for all $j\geq j(n)$ and $n\geq N$. Finally, note that $\pi_{D_j}(\Gamma(j,n))\subset \hat f_j(\mathcal G^*(z_n,q_j;U))$. This completes the proof of \eqref{theorem:general_caratheodory:loewner}.
\end{proof}

We will apply this theorem in the special case that $D_j$ is a circle domain for each $j\in \N$. The convergence of $\{f_j\}_{j\in \N}$ to $f$ in condition \ref{theorem:general_caratheodory:3} can be readily obtained by choosing appropriate normalizations and passing to a subsequence. Instead, establishing condition \ref{theorem:general_caratheodory:2} is considerably more technical. We provide sufficient conditions for \ref{theorem:general_caratheodory:2} to be true in the next section. 

We include an elementary lemma for future reference.

\begin{lemma}\label{lemma:domains_converge}
Let $\Omega\subset \widehat \C$ be a domain. Suppose that the following conditions hold.
\begin{enumerate}[label=\normalfont(\arabic*)]
	\item\label{lemma:domains_converge:1} There exists a sequence of domains $\{\Omega_j\}_{j\in \N}$ such that each compact subset of $\Omega$ is contained in $\Omega_j$ for all sufficiently large $j\in \N$.
	\item\label{lemma:domains_converge:2} For each $j\in \N$, there exists a conformal map $f_j$ from $\Omega_j$ onto a domain $D_j\subset \widehat \C$ such that $\{f_j\}_{j\in \N}$ converges locally uniformly in $\Omega$ to a conformal map $f$ from $\Omega$ onto a domain $D\subset \widehat \C$.
\end{enumerate}
Then
\begin{enumerate}[label=\normalfont(\arabic*)]\setcounter{enumi}{2}
	\item\label{lemma:domains_converge:3} each Hausdorff limit of a sequence $p_j\in \mathcal C(\Omega_j)$, $j\in \N$, is contained in some $p\in \mathcal C(\Omega)$ and
	\item\label{lemma:domains_converge:4} conditions \ref{lemma:domains_converge:1} and \ref{lemma:domains_converge:3} hold for the domain $D$ and the sequence $\{D_j\}_{j\in \N}$. 	
\end{enumerate}
\end{lemma}

\begin{proof}
Let $E$ be the Hausdorff limit of a sequence $p_{j_k}\in \mathcal C(\Omega_{j_k})$, $k\in \N$. If there exists a point $z_0\in E\cap \Omega$, then there exists a compact neighborhood $U$ of $z_0$ that lies in $\Omega$ and hence in $\Omega_j$ for all sufficiently large $j\in \N$ by assumption \ref{lemma:domains_converge:1}. However, $U\cap p_{j_k}\neq \emptyset$ for all sufficiently large $k\in \N$, a contradiction. Thus, $E\subset \widehat \C\setminus \Omega$. Since $E$ is connected, we have $E\subset p$ for some $p\in \mathcal C(\Omega)$.

For \ref{lemma:domains_converge:4} it suffices to verify \ref{lemma:domains_converge:1} for $D$ and $\{D_j\}_{j\in \N}$. Let $z_0\in \Omega$ and  $w_0=f(z_0)\in D$. Consider a ball $B \subset \bar B\subset \Omega$ centered at $z_0$, so $\bar B\subset \Omega_j$ for all sufficiently large $j\in \N$. By the argument principle and the uniform convergence of $\{f_j|_{\bar B}\}_{j\in \N}$ to $f|_{\bar B}$, there exists a neighborhood $U$ of $w_0$ that is contained in $f_j(\bar B)$ for all sufficiently large $j\in \N$. Thus, $U\subset D_j$ for all sufficiently large $j\in \N$. Now, given a compact set $K\subset D$, we can cover it by finitely many such neighborhoods, so we have $K\subset D_j$ for all sufficiently large $j\in \N$, as desired. 
\end{proof}

\subsection{Potential theoretic conditions for uniformization}
In this section we use the spherical metric. For $z_0\in \widehat \C$ and $r>0$ we denote by $B(z_0,r)$ the spherical ball $\{z\in \widehat \C: \sigma(z,z_0)<r\}$. For $0<r<R<\pi$ we let $S(z_0,r)$ be the spherical circle $\{z\in \widehat{\C}: \sigma(z,z_0)=r\}$ and $A(z_0;r,R)$ be the annulus bounded by $S(z_0,r)$ and $S(z_0,R)$.

For sets $E,F,G$ in a topological space $X$ we denote by $\Gamma(E,F;G)$ the family of curves $\gamma\colon [a,b]\to  G$ such that $\gamma(a)\in E$ and $\gamma(b)\in F$. Also, we denote by $\Gamma^*(E,F;G)$ the family of curves $\gamma\colon [a,b]\to G$ that separate $E$ and $F$, that is, $E$ and $F$ lie in distinct components of $X\setminus \gamma([a,b])$.

\begin{lemma}\label{lemma:potential}
Suppose that $\Omega,q$ and $\Omega_j,q_j$, $j\in \N$, are as in Theorem \ref{theorem:general_caratheodory} \ref{theorem:general_caratheodory:1} and that for each $z_0\in \partial q$ and $R\in (0,\pi)$ we have 
$$ \lim_{r\to 0^+} \limsup_{j\to\infty}\Mod_{\Omega_j} \Gamma(\pi_{\Omega_j}(S(z_0,r)), \pi_{\Omega_j}(S(z_0,R)); \hat\Omega_j\setminus \{q_j\}) =0.$$
Then condition \ref{theorem:general_caratheodory:2} of Theorem \ref{theorem:general_caratheodory} is satisfied.
\end{lemma}

\begin{proof}
Consider a Jordan region $U\supset q$ with $\partial U\subset \Omega$. Also, let $z_n\in \Omega$, $n\in \N$, such that $z_n\to z_0$ for some $z_0\in \partial q$. Let  $R>0$ such that $\bar B(z_0,R)\subset U$ and let $r\in (0,R)$. Let $n_0\in \N$ such that $z_n\in B(z_0,r)$ for $n\geq n_0$. We fix $n\geq n_0$. For sufficiently large $j\in \N$, say for $j\geq j(n)$, we have $z_n\in \Omega_j$, so $z_n\notin q_j$. We fix $j\geq j(n)$.

We claim that each curve in $\mathcal G^*(z_n,q_j;U)$ has a subcurve in 
$$\Gamma(r,j)=\Gamma(\pi_{\Omega_j}(S(z_0,r)), \pi_{\Omega_j}(S(z_0,R)); \hat\Omega_j\setminus \{q_j\}).$$
Let $\gamma\in \mathcal G^*(z_n,q_j;U)$. Then $\pi_{\Omega_j}^{-1}(|\gamma|)$ is a continuum that intersects $\partial U$, so it is not contained in $B(z_0,R)$. Also, $\pi_{\Omega_{j}}^{-1}(|\gamma|)$ intersects $B(z_0,r)$, because $\gamma$ separates $z_n$ and $q_j$ in $\pi_{\Omega_j}(\bar U)$. Thus, $\pi_{\Omega_j}^{-1}(|\gamma|)$ intersects both $S(z_0,r)$ and $S(z_0,R)$. Equivalently, $|\gamma|$ intersects both $\pi_{\Omega_j}(S(z_0,r))$ and $\pi_{\Omega_j}(S(z_0,R))$. The claim follows. Therefore, for all $j\geq j(n)$ we have
$$\Mod_{\Omega_j}\mathcal G^*(z_n,q_j;U) \leq \Mod_{\Omega_j}\Gamma(r,j)$$
This implies that 
$$\limsup_{j\to\infty }\Mod_{\Omega_j} \mathcal G^*(z_n,q_j;U) \leq \limsup_{j\to\infty}\Mod_{\Omega_j}\Gamma(r,j)$$
for all $n\geq n_0$. By assumption, the right-hand side converges to $0$ as $r\to 0^+$. Hence, letting $n\to \infty$ and then $r\to 0^+$ shows that
$$\lim_{\substack{z\to  q\\z\in \Omega}} \limsup_{j\to\infty} \Mod_{\Omega_j} \mathcal G^*(z,q_j; U)=0.$$
Thus, condition \ref{theorem:general_caratheodory:2} of Theorem \ref{theorem:general_caratheodory} holds. 
\end{proof}

\begin{theorem}\label{theorem:koebe_nondegenerate}
In the setting of Theorem \ref{theorem:general_caratheodory}, if the condition of Lemma \ref{lemma:potential} is satisfied for $q$, then $q$ is a singleton if and only if $\hat f(q)$ is a singleton.
\end{theorem}

The proof resembles the proof of the last part of \cite{EsmayliRajala:quasitripod}*{Theorem 1.3}.

\begin{proof}
Let $f_j\colon \Omega_j\to D_j$, $j\in \N$, be a sequence of conformal maps that converges locally uniformly in $\Omega$ to a conformal map $f\colon \Omega\to D$, as in Theorem \ref{theorem:general_caratheodory} \ref{theorem:general_caratheodory:3}. If $E$ is a Hausdorff limit of $\{\hat f_j(q_j)\}_{j\in \N}$, then the set $\hat f(q)$ is the complement of the component of $\widehat \C\setminus E$ that contains $D$, as in the conclusion of Theorem \ref{theorem:general_caratheodory}. In order to prove the current theorem, it suffices to show that $q$ is a singleton if and only if $E$ is a singleton.  Let $U\supset q$ be a Jordan region such that $\partial U\subset \Omega$ as in Theorem \ref{theorem:general_caratheodory} \ref{theorem:general_caratheodory:2}. 

Suppose that $E$ is non-degenerate. As $j\to\infty$ we have
\begin{align*}
\Delta(\hat f_j(q_j), f_j(\partial U))=\frac{\dist(\hat f_j(q_j), f_j(\partial U))}{\min\{\diam \hat f_j(q_j), \diam f_j(\partial U)\}} \to \frac{\dist(E,f(\partial U))}{\min\{\diam E,\diam f(\partial U)\}}
\end{align*}
so there exists $\Delta>0$ such that $\Delta(\hat f_j(q_j), f_j(\partial U))\leq \Delta$ for each $j\in \N$. By the Loewner property of $\widehat \C$ \cite{HeinonenKoskela:qc}*{Section 6}, there exists $\delta>0$ such that 
$$\Mod \Gamma( \hat f_j(q_j), f_j(\partial U);\widehat \C) \geq \delta\,\,\, \text{for each $j\in \N$.}$$
By the modulus comparison given in Proposition \ref{proposition:comparison}, we have
$$\Mod_{D_j} \Gamma( \hat f_j(q_j), f_j(\partial U);\hat D_j)\geq  c(\tau)^{-1}\min\{1,\delta\}\eqqcolon C_0\,\,\, \text{for each $j\in \N$.}$$
The conformal invariance of transboundary modulus implies that
$$\Mod_{\Omega_j} \Gamma( q_j, \partial U;\hat \Omega_j) \geq C_0\,\,\, \text{for each $j\in \N$.}$$
Suppose that $q$ consists of a single point $z_0$ and let $0<r<R$ such that $\bar B(z_0,R)\subset U$. By Theorem \ref{theorem:general_caratheodory} \ref{theorem:general_caratheodory:1} we have $q_j\to q$ as $j\to\infty$, so for sufficiently large $j\in \N$ we have $q_j\subset B(z_0,r)$, say for $j\geq j_0$. We note that each curve in $\Gamma( q_j, \partial U;\hat \Omega_j)$ contains a subcurve in 
$$\Gamma(r,j)=\Gamma(\pi_{\Omega_j}(S(z_0,r)), \pi_{\Omega_j}(S(z_0,R));  \hat \Omega_j\setminus \{q_j\}).$$
Thus, for $j\geq j_0$ we have
\begin{align*}
\Mod_{\Omega_j}&\Gamma(\pi_{\Omega_j}(S(z_0,r)), \pi_{\Omega_j}(S(z_0,R));  \hat \Omega_j\setminus \{q_j\})\geq \Mod_{\Omega_j} \Gamma( q_j, \partial U;\hat \Omega_j) \geq C_0.
\end{align*}
Letting $j\to\infty$ and then $r\to 0^+$ contradicts the condition of Lemma \ref{lemma:potential}. Therefore, $q$ is non-degenerate.

%We now prove the claim. Let $\gamma\in \Gamma(q_j,\partial U;\hat \Omega_j)$. Let $\phi_k\colon \widehat \C\to \hat\Omega_j$, $k\in \N$, be a sequence of homeomorphisms converging uniformly to $\pi_{\Omega_j}$. Then the Jordan curve $\phi_k(S(z_0,r))$ separates $\phi_k(q_j)$ from $\phi_k(\partial U)$ and thus from $\pi_{\Omega_j}(\partial U)$ for all sufficiently large $k\in \N$. Suppose that $q_j$ and $\phi_k(q_j)$ lie in distinct components of $\hat \Omega_j\setminus \phi_k(S(z_0,r))$ for all sufficiently large $k\in \N$. In the limit, since $\phi_k(q_j)\to q_j$ as $k\to\infty$, we see that $q_j\in \pi_{\Omega_j}(S(z_0,r))$. This implies that $q_j\cap S(z_0,r)\neq \emptyset$ in $\widehat \C$, which is a contradiction since $q_j\subset B(z_0,r)$. Thus, for large enough $k\in \N$ the point $q_j$ and the set $\phi_k(q_j)$ lie in the same complementary component of $\phi_k(S(z_0,r))$. Since $\gamma$ connects $q_j$ and $\pi_{\Omega_j}(\partial U)$, we conclude that $\gamma$ intersects $\phi_k(S(z_0,r))$ and, similarly, it intersects $\phi_k(S(z_0,R))$. It follows that $\gamma$ has a subcurve $\gamma_k$ with endpoints in $\phi_k(S(z_0,r))$ and $\phi_k(S(z_0,R))$ that is disjoint from $q_j$. Passing to the limit, we see that $\gamma$ has a subcurve $\gamma'$ with endpoints in $\pi_{\Omega_j}(S(z_0,r))$ and $\pi_{\Omega_j}(S(z_0,R))$ such that $q_j$ is not an interior point of $\gamma$. Note that $q_j$ cannot be an endpoint of $\gamma'$ either, since $q_j\subset B(z_0,r)$. Therefore, $\gamma'\in  \Gamma(r,j)$, as claimed.

Next, suppose that $E$ is a single point $w_0$. By passing to a subsequence, suppose that $\hat f_j(q_j)\to w_0$ as $j\to\infty$. We conclude that $\hat f(q)=w_0$ is a complementary component of $D$. Consider $R_1>0$ such that the ball $\bar B(w_0,R_1)$ is contained in the component of $\widehat \C\setminus f(\partial U)$ that contains $w_0$. We let $A_1=A(w_0;R_1/2,R_1)$. Since $w_0\in \mathcal C(D)$, there exists $R_2<R_1/2$ such that no $p\in \mathcal C(D)$ intersects both annuli $\bar A_1$ and $\bar A_2$, where $A_2=A(w_0;R_2/2,R_2)$.  Inductively, for each $i\in \N$ we may find an annulus $A_i=A(w_0;R_i/2,R_i)$ with $R_{i+1}<R_i/2$ for $i\in \N$ such that no $p\in \mathcal C(D)$ intersects both $\bar A_i$ and $\bar A_{i+1}$. As a consequence, for each fixed $i\in \N$ and for all sufficiently large $j\in \N$ no component $p\in \mathcal C(D_j)$ intersects both annuli $\bar A_i$ and $\bar A_{i+1}$. Indeed, if there exists $p_j\in \mathcal C(D_j)$ that intersects $\bar A_i$ and $\bar A_{i+1}$ for infinitely many $j\in \N$, then there exists a Hausdorff limit $F$ of $\{p_j\}_{j\in \N}$ that intersects both $\bar A_i$ and $\bar A_{i+1}$. By Lemma \ref{lemma:domains_converge} \ref{lemma:domains_converge:4}, there exists $p\in \mathcal C(D)$ such that $F\subset p$. Thus, $p$ intersects both $\bar A_i$ and $\bar A_{i+1}$, a contradiction.

Let $N\in \N$. For all sufficiently large $j\in\N$, no $p\in \mathcal C(D_j)$ intersects two of the annuli $A_i$, $i\in \{1,\dots,N\}$. Also, there exists $C_0>0$ such that
$$\Mod \Gamma^*(A_i) = \frac{1}{2\pi} \log \frac{\tan(R_i/2)}{\tan(R_i/4)}\geq C_0 \,\,\, \text{for all $i\in \{1,\dots,N\}$}.$$
By the modulus comparison from Proposition \ref{proposition:comparison}, we have 
$$\Mod_{D_j}\pi_{D_j}(\Gamma^*(A_i)) \geq c(\tau)^{-1} \min\{ 1, C_0\}\eqqcolon C_1\,\,\, \text{for all $j\in \N$}.$$
For all sufficiently large $j\in \N$ we have $\hat f_j(q_j) \subset B(w_0, R_N/2)$. This implies that $\pi_{D_j}(\Gamma^*(A_i))\subset \Gamma^*(\hat f_j(q_j),  f_j(\partial U);\hat D_j)$ for each $i\in \{1,\dots,N\}$. Let $\rho\colon \hat D_j\to [0,\infty]$ be admissible for the transboundary modulus of $\Gamma^*(\hat f_j(q_j),  f_j(\partial U);\hat D_j)$ and for $i\in \{1,\dots,N\}$ we define 
$$\widetilde \rho_i(z)=\begin{cases}\rho(z)\cdot \chi_{A_i\cap D_j}(z),& z\in D_j\\
\rho(z), & \text{$z=p$ for some $p\in \mathcal C(D_j)$ with $p\cap A_i\neq \emptyset$}\\
0,& \text{$z=p$ for some $p\in \mathcal C(D_j)$ with $p\cap A_i=\emptyset$}
\end{cases}.$$
Then $\widetilde \rho_i$ is admissible for $\pi_{D_j}(\Gamma^*(A_i))$ for each $i\in \{1,\dots,N\}$. Since the annuli $A_1,\dots,A_N$ are disjoint and each component $p\in \mathcal C(D_j)$ can intersect at most one of these annuli for sufficiently large $j\in \N$, we obtain
\begin{align*}
\int_{D_j} \rho^2\, d\Sigma+\sum_{p\in \mathcal C(D_j)}\rho(p)^2 &\geq \sum_{i=1}^N \left(\int_{D_j} \widetilde\rho_i^2\, d\Sigma+\sum_{p\in \mathcal C(D_j)} \widetilde \rho_i(p)^2\right) \\
&\geq \sum_{i=1}^N\Mod_{D_j}\pi_{D_j}(\Gamma^*(A_i))\geq C_1N.
\end{align*}
Infimizing over $\rho$ gives
$$\Mod_{D_j}\Gamma^*(\hat f_j(q_j),  f_j(\partial U); \hat D_j) \geq C_1N.$$ 
By letting $j\to\infty$ and then $N\to\infty$, we obtain
$$\lim_{j\to\infty}\Mod_{D_j}\Gamma^*(\hat f_j(q_j),  f_j(\partial U);\hat D_j)=\infty.$$
By the conformal invariance of transboundary modulus, we have
\begin{align}\label{theorem:koebe_nondegenerate:infinity}
\lim_{j\to\infty}\Mod_{\Omega_j}\Gamma^*(q_j,  \partial U;\hat \Omega_j)=\infty.
\end{align}

For the sake of contradiction, suppose that $q$ is non-degenerate. Let $z_0\in \partial q$ and fix $R>0$ such that $\bar B(z_0,R)\subset U$ and $q\not\subset B(z_0,R)$. Also, by the condition of Lemma \ref{lemma:potential}, there exist $r\in (0,R)$ and $j_0\in \N$ such that 
$$\Mod_{\Omega_j} \Gamma(j)<1\,\,\, \text{for $j\geq j_0$},$$
where $\Gamma(j)=\Gamma(\pi_{\Omega_j}(S(z_0,r)), \pi_{\Omega_j}(S(z_0,R));  \hat \Omega_j\setminus \{q_j\})$. If $S(z_0,r)\subset \widehat \C\setminus \Omega$, then $S(z_0,r)\subset q$, so $\bar B(z_0,r)\subset q$, which contradicts the assumption that $z_0\in \partial q$. Thus $S(z_0,r)\cap \Omega\neq \emptyset$. Consider disjoint arcs $S,T$ in $\Omega$ connecting $S(z_0,r)$ and $\partial U$. Also, let $\delta>0$ such that $\bar{N_\delta(T)}\subset \Omega$ and $S\cap {N_\delta(T)}= \emptyset$, where
$$N_\delta(T)= \{z\in \widehat \C: \dist(z,T)<\delta\}.$$
See Figure \ref{fig:potential} for an illustration. For all sufficiently large $j\in \N$, we have $S\cup \bar {N_\delta(T)}\subset \Omega_j$ by Theorem \ref{theorem:general_caratheodory} \ref{theorem:general_caratheodory:1}. Let $\rho_j$ be an admissible density for $\Gamma(j)$ with total mass bounded above by $1$. We define $\widetilde \rho_j=\rho_j$ in the complement of $\Omega_j$ and $\widetilde \rho_j=\rho_j+ \delta^{-1}\chi_{N_\delta(T)}$ in $\Omega_j$. Note that 
$$\int_{\Omega_j} \widetilde \rho_j^2\, d\Sigma+ \sum_{p\in \mathcal C(\Omega_j)}\widetilde \rho_j(p)^2\leq (1+\delta^{-1}\Sigma(N_\delta(T))^{1/2})^2\eqqcolon C_2$$
for all sufficiently large $j\in \N$. We claim that $\widetilde \rho_j$ is admissible for $\Gamma^*(q_j,  \partial U;\hat \Omega_j)$,  hence
$$\Mod_{\Omega_j}\Gamma^*(q_j,  \partial U;\hat \Omega_j)\leq C_2 $$
for all sufficiently large $j\in \N$. This contradicts \eqref{theorem:koebe_nondegenerate:infinity}, so $q$ must be a singleton.

\begin{figure}
\centering
	\begin{tikzpicture}
	
	\draw[rounded corners=10pt, line width=0.4cm, line cap=round, red!30] (-0.2,0.45)--(-0.5,1.3)--(0.5,1)--(1,2.2);
	\draw[rounded corners=10pt, thick, red] (-0.2,0.45)--(-0.5,1.3)--(0.5,1)--(1,2.2);
	\node at (0.1,1.5) {$\scriptstyle N_\delta(T)$};
	
	\draw[rounded corners, thick, red] (0.3,-0.4)--(0.8,-0.3)--(1.4,-0.8)--(2.2,0)--(3,-0.3);
	\node at (2.5,0.15) {$S$};
	
	\draw[line width=0.34cm, line cap=round] (-3,0)--(1,0);
	\draw[line width=0.3cm, line cap=round,color=black!20] (-3,0)--(1,0);
	\draw[line width=1pt, line cap=round] (-3,0)--(1,0);
	\node at (-3,-0.4) {$q_j$};
	
	\fill (0,0) circle (1.5pt);
	\draw[dashed] (0,0) circle (0.5cm);
	\draw[dashed] (0,0) circle (2cm);
	\node at (0,-0.7) {$\scriptstyle S(z_0,r)$};
	\node at (0,-1.8) {$\scriptstyle S(z_0,R)$};
	
	\draw[dotted, rounded corners, thick] (-4,-2.2) rectangle (3,2.2);
	\node at (3.4,-2) {$\partial U$};
	\end{tikzpicture}
	\caption{Two arcs $S,T\subset \Omega$ connecting $S(z_0,r)$ and $\partial U$ such that $S\cap N_\delta(T)=\emptyset$. }\label{fig:potential}
\end{figure}
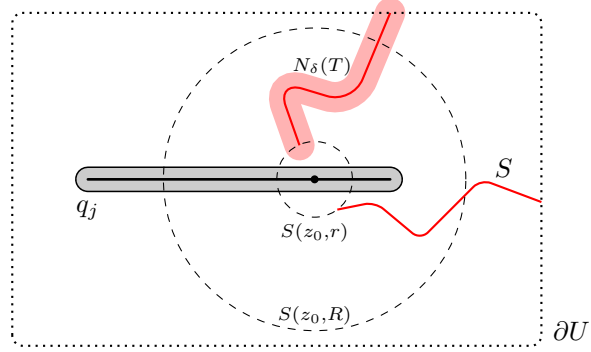

We now prove our claim. Let $\gamma\in \Gamma^*(q_j,\partial U;\hat \Omega_j)$, so $\pi_{\Omega_j}^{-1}(|\gamma|)$ is a continuum that separates $q_j$ from $\partial U$. Since $q\not\subset B(z_0,R)$, we have $q_j\not\subset B(z_0,R)$ and $\pi_{\Omega_j}^{-1}(|\gamma|)\not \subset B(z_0,R)$. If $\pi_{\Omega_j}^{-1}(|\gamma|)$ intersects $\bar B(z_0,r)$, then $|\gamma|$ intersects both $\pi_{\Omega_j}(S(z_0,r))$ and $\pi_{\Omega_j}(S(z_0,R))$ so $\gamma$ has a subpath in $\Gamma(j)$ and 
$$\int_{\gamma} \rho_j\, ds+ \sum_{\substack{p\in \mathcal C(\Omega_j)\\ |\gamma|\cap p\neq \emptyset}}\rho_j(p)\geq 1.$$
Suppose that $\pi_{\Omega_j}^{-1}(|\gamma|)$ is disjoint from $\bar B(z_0,r)$. 
Note that $\bar B(z_0,r)\cup S$ and  $\bar B(z_0,r)\cup T$ are continua that connect $q_j$ and $\partial U$. Thus, $\pi_{\Omega_j}^{-1}(|\gamma|)$ must intersect both $S$ and $T$. Since $S\cap N_\delta(T)=\emptyset$, the set $\pi_{\Omega_j}^{-1}(|\gamma|)$ connects $T$ to $\partial N_\delta(T)$, so we have
$$ \int_\gamma \delta^{-1}\chi_{N_\delta(T)}\, ds \geq 1.$$
In both cases, 
$$\int_{\gamma} \widetilde \rho_j\, ds+ \sum_{\substack{p\in \mathcal C(\Omega_j)\\ |\gamma|\cap p\neq \emptyset}}\widetilde \rho_j(p)\geq 1,$$
which completes the proof of the admissibility.
\end{proof}

\subsection{Verification of main condition}
We provide sufficient conditions so that the main assumption of Lemma \ref{lemma:potential} holds. We continue to use the spherical metric.

\begin{lemma}\label{lemma:potential_sufficient}
Suppose that $\Omega,q$ and $\Omega_j,q_j$, $j\in \N$, are as in Theorem \ref{theorem:general_caratheodory} \ref{theorem:general_caratheodory:1} and that the following conditions hold.
\begin{enumerate}[label=\normalfont(\arabic*)]
\item\label{lemma:potential_sufficient:1} For each $\varepsilon>0$ the family of components $p\in \mathcal C(\Omega)$ with spherical diameter greater than $\varepsilon$ is finite.
\item\label{lemma:potential_sufficient:3} For each $z_0\in \partial q$ there exists a sequence of annuli $\{A_i=\widehat \C\setminus (E_i\cup F_i)\}_{i\in \N}$ that separate $z_0$ from a fixed point of $\Omega$, converge to $z_0$ in the Hausdorff sense, and
$$\sup_{i\in \N} \left( \limsup_{j\to\infty}\Mod_{\Omega_j} \Gamma(\pi_{\Omega_j}(E_i),\pi_{\Omega_j}(F_i); \pi_{\Omega_j}(\bar A_i)) \right)<\infty.$$ 
\end{enumerate}
Then the condition of Lemma \ref{lemma:potential} is satisfied.
\end{lemma}

In the notation $A_i=\widehat \C\setminus (E_i\cup F_i)$ it is implicitly understood that $E_i$ and $F_i$ are the two complementary components of the annulus $A_i$.

\begin{proof}
Let $R\in (0,\pi)$, $z_0\in \partial q$, and $N\in \N$. By assumption \ref{lemma:potential_sufficient:3}, there exist annuli $A_{i}$, $i\in \N$, that separate $z_0$ from a fixed base point $z_\infty\in \Omega$,  shrink to $z_0$ as $i\to\infty$, and satisfy
\begin{align}\label{assumption:pointwise}
\limsup_{j\to\infty}\Mod_{\Omega_j} \Gamma(i,j) < C(z_0)\,\,\, \text{for all $i\in \N$},
\end{align}
where $\Gamma(i,j)=\Gamma(\pi_{\Omega_j}(E_i),\pi_{\Omega_j}(F_i); \pi_{\Omega_j}(\bar A_i))$. We let $i_1\in \N$ such that $A_{i_1}\subset B(z_0,R)$. Suppose that the index $i_{k-1}$ has been defined for some $k\geq 2$. Since $A_i$ converges to $z_0$ as $i\to\infty$, there exists $i_k\in \N$, $i_k>i_{k-1}$, such that $\bar A_{i_k}$ is contained in the complementary component of $A_{i_{k-1}}$ that contains $z_0$. Moreover, by \ref{lemma:potential_sufficient:1}, at most finitely many components $p\in \mathcal C(\Omega)$ can intersect both $A_{i_k}$ and $ A_{i_{k-1}}$. By choosing $A_{i_{k}}$ to be even smaller, we may further have that no component $p\in \mathcal C(\Omega)$ other than $q$ intersects both $ A_{i_k}$ and $ A_{i_{k-1}}$. Let $\delta>0$ such that $B(z_0,\delta)$ is contained in the complementary component of $A_{i_N}$ that contains $z_0$.

For sufficiently large $j\in \N$, depending on $N$, by \eqref{assumption:pointwise}, there exists an admissible function $\rho_{j,k}\colon \hat \Omega_j\to[0,\infty]$ for $\Gamma(i_k,j)$ such that
$$\int_{\Omega_j} \rho_{j,k}^2\, d\Sigma + \sum_{p\in \mathcal C(\Omega_j)}\rho_{j,k}(p)^2<C(z_0)\,\,\, \text{for each $k=1,\dots,N$.}$$
For $k\in \{1,\dots,N\}$ we set $\widetilde \rho_{j,k}=\rho_{j,k}\cdot \chi_{A_{i_k}}$ in $\Omega_j$, $\widetilde \rho_{j,k}(q_j)=0$, $\widetilde \rho_{j,k}(p)=0$ for $p\in \mathcal C(\Omega_j)$ such that $p\cap A_{i_k}=\emptyset$, and $\widetilde \rho_{j,k}(p)=\rho_{j,k}(p)$ otherwise. The function $\rho=\frac{1}{N}\sum_{k=1}^N \widetilde\rho_{j,k}$ is admissible for  $\Gamma(\pi_{\Omega_j}(S(z_0,r)), \pi_{\Omega_j}(S(z_0,R)); \hat \Omega_j\setminus \{q_j\})$ for every $r\in (0,\delta)$. 

Suppose that for infinitely many $j\in \N$ there exists a component $p_{j}\in \mathcal C(\Omega_{j})$ that intersects two of the annuli $A_{i_k}$, $k\in \{1,\dots,N\}$. Then, by Lemma \ref{lemma:domains_converge} \ref{lemma:domains_converge:3}, there exists $p\in \mathcal C(\Omega)$ intersecting two such annuli. This is a contradiction. Thus, for all sufficiently large $j\in \N$, each component of $\mathcal C(\Omega_j)$ intersects at most one of the annuli $A_{i_k}$, $k\in \{1,\dots,N\}$. It follows that
$$\int_{\Omega_j}\rho^2\, d\Sigma +\sum_{p\in \mathcal C(\Omega_j)}\rho(p)^2\leq  \frac{1}{N^2} \sum_{k=1}^N \left(\int_{\Omega_j} \rho_{j,k}^2\, d\Sigma + \sum_{p\in \mathcal C(\Omega_j)}\rho_{j,k}(p)^2\right)\leq \frac{C(z_0)}{N}$$
for all sufficiently large $j\in \N$. We conclude that 
$$\Mod_{\Omega_j}\Gamma(\pi_{\Omega_j}(S(z_0,r)), \pi_{\Omega_j}(S(z_0,R)); \hat \Omega_j\setminus \{q_j\})\leq \frac{C(z_0)}{N}$$
for every $r\in (0,\delta)$ and for all sufficiently large $j\in \N$. By letting $j\to\infty$, $r\to 0^+$, and $N\to\infty$, we obtain the condition of Lemma \ref{lemma:potential}, as desired.
\end{proof}

In the following notions we use the Euclidean metric. Recall from Section \ref{section:john} the definitions of $\mathbb A(z_0;r,R)$, $\D(z_0,r)$, and $\mathbb S(z_0,r)$ for $z_0\in \C$ and $0<r<R$. Moreover, we set $\mathbb A(\infty;r,R)=\{z\in \C: R^{-1}<|z|<r^{-1}\}=\mathbb A(0;R^{-1},r^{-1})$, $\mathbb D(\infty,r)=\{z\in \C: |z|>r^{-1}\}\cup \{\infty\}$, and $\mathbb S(\infty,r)=\partial \D(\infty,r)=\mathbb S(0,r^{-1})$. Two-dimensional Lebesgue measure in $\C$ is denoted by $m_2$.

For a set $E\subset \widehat \C$ and an annulus $A=\mathbb A(z_0;r,R)$, $z_0\in \C$, such that $E\cap A\neq \emptyset$ we define $r_A(E)=\inf\{|z_0-z|: z\in A\cap E\}$ and $R_A(E)=\sup\{|z_0-z|:z\in A\cap E\}$. If $A=\mathbb A(\infty;r,R)=\mathbb A(0;R^{-1},r^{-1})$, then $r_A(E)=\inf\{|z|:z\in A\cap E\}$ and $R_A(E)=\sup\{|z|: z\in A\cap E\}$.  We also define the \textit{width of $E$ relative to the annulus $A$} as 
$$w_A(E)=\log \frac{R_A(E)}{r_A(E)}.$$
If $E\cap A=\emptyset$ we set $w_A(E)=0$.

\begin{lemma}\label{lemma:diameter_liminf}
Let $\Omega\subset \widehat \C$ be a countably connected domain. For $z_0\in \widehat \C$, $0<r<R$, and $A=\mathbb A(z_0;r,R)$ we have
\begin{align}\label{lemma:diameter_liminf:claim}
\begin{aligned}
\Mod_{\Omega}  \Gamma( \pi_{\Omega}(\mathbb S(z_0,r)),&\pi_{\Omega}(\mathbb S(z_0,R)); \pi_{\Omega}(\bar A)) \\
&\leq 2\pi\bigg(\log\frac{R}{r}\bigg)^{-1} +\bigg(\log\frac{R}{r}\bigg)^{-2}\sum_{\substack{p\in \mathcal C(\Omega)\\ p\cap A\neq \emptyset}} w_{A} (p)^2.
\end{aligned}
\end{align}
\end{lemma}

\begin{proof}
If $z_0\in \C$, we set 
$$\rho(z)= \frac{\chi_{A}(z)}{\log(R/r)\cdot |z-z_0|}\,\, \text{for $z\in \Omega$} \quad \text{and} \quad \rho(p)=\frac{w_{A}(p)}{\log (R/r)}\,\, \text{for $p\in \mathcal C(\Omega)$}.$$
If $z_0=\infty$, we replace $|z-z_0|$ with $|z|$ in the first of the above formulas. Since $\Omega$ is countably connected, one can show that $\rho$ is admissible for the transboundary modulus of $\Gamma( \pi_{\Omega}(\mathbb S(z_0,r)),\pi_{\Omega}(\mathbb S(z_0,R)); \pi_{\Omega}(\bar A))$. The argument can be found in the proof of \cite{Ntalampekos:schottky}*{Prop.\ 5.8} or \cite{Bonk:uniformization}*{Prop.\ 8.7}, so we omit it. The total mass of $\rho$ is exactly the quantity in the right hand side of \eqref{lemma:diameter_liminf:claim}.
\end{proof}

We now combine the previous lemmas.

\begin{proposition}\label{prop:geometric_sufficient}
Suppose that $\Omega,q$ and $\Omega_j,q_j$, $j\in \N$, are as in Theorem \ref{theorem:general_caratheodory} \ref{theorem:general_caratheodory:1} and that the following conditions hold.
\begin{enumerate}[label=\normalfont(\arabic*)]
\item\label{prop:geometric_sufficient:1} For each $\varepsilon>0$ the family of components $p\in \mathcal C(\Omega)$ with spherical diameter greater than $\varepsilon$ is finite.
\item\label{prop:geometric_sufficient:2} $\Omega_j$ is countably connected for each $j\in \N$. 
\item \label{prop:geometric_sufficient:3} For each $z_0\in \partial q$ and for $A_r=\mathbb A(z_0;r/2,r)$, $r>0$, we have
$$\liminf_{r\to 0^+}\limsup_{j\to\infty}\sum_{\substack{p\in \mathcal C(\Omega_j)\\ p\cap A_r\neq \emptyset}} w_{A_r} (p)^2 <\infty.$$
\end{enumerate}
Then the condition of Lemma \ref{lemma:potential} is satisfied.
\end{proposition}

\begin{proof}
By Lemma \ref{lemma:diameter_liminf} and assumptions \ref{prop:geometric_sufficient:2}, \ref{prop:geometric_sufficient:3}, for each $z_0\in \partial q$ we have
$$\liminf_{r\to0^+}\limsup_{j\to\infty}\Mod_{\Omega_j}  \Gamma( \pi_{\Omega_j}(\mathbb S(z_0,r/2)),\pi_{\Omega}(\mathbb S(z_0,r)); \pi_{\Omega_j}(\bar A_r))<\infty.$$
Thus, there exists a sequence $r_i\to 0^+$ such that
$$ \sup_{i\in \N}\left(\limsup_{j\to\infty}\Mod_{\Omega_j}  \Gamma( \pi_{\Omega_j}(\mathbb S(z_0,r_i/2)),\pi_{\Omega}(\mathbb S(z_0,r_i)); \pi_{\Omega_j}(\bar A_{r_i}))\right)<\infty,$$
which implies Lemma \ref{lemma:potential_sufficient} \ref{lemma:potential_sufficient:3}. Lemma \ref{lemma:potential_sufficient}  \ref{lemma:potential_sufficient:1} also holds by assumption \ref{prop:geometric_sufficient:1}. By Lemma \ref{lemma:potential_sufficient}, the condition of Lemma \ref{lemma:potential} holds for $q$. 
\end{proof}

\subsection{Geometric conditions for uniformization}\label{section:geometric}
In this section we use the the Euclidean metric.

\begin{lemma}\label{lemma:characteristics}
Let $\Omega\subset \widehat \C$ be a domain and $A=\mathbb A(z_0;t/2,t)$ be an annulus, where $z_0\in \widehat \C$ and $t>0$, satisfying the following conditions for some $L\geq 1$.
\begin{enumerate}[label=\normalfont(\arabic*)]
	\item\label{lemma:characteristics:1} The number of components $p\in \mathcal C(\Omega)$ that intersect both complementary components of the annulus $\mathbb A(z_0;t,2t)$ is bounded above by $L$. 
	\item\label{lemma:characteristics:2} There exists a set $\mathcal C\subset \mathcal C(\Omega)$ with $\mathcal C(\Omega)\setminus \mathcal C$ having at most $L$ non-degenerate elements such that for each $p\in \mathcal C$ there exists a ball $B_p\subset \C$ with
\begin{align*}
\diam B_p \leq L\diam p,\quad   p\subset LB_p,\quad \text{and}\quad \sum_{p\in \mathcal C}\chi_{B_p}\leq L.
\end{align*}
\end{enumerate}
Then 
$$\sum_{\substack{p\in \mathcal C(\Omega)\\ p\cap A\neq \emptyset}} w_{A} (p)^2 \leq C(L).$$
\end{lemma}

\begin{proof}
Consider the annulus $A=\mathbb A(z_0;t/2,t)$. If $z_0=\infty$, then $A= \mathbb A(\infty;t/2,t)=\mathbb A(0;t^{-1},2t^{-1})$, so it suffices to consider the case that $z_0\in \C$. 

Let $p\in \mathcal C$ such that $p\cap A\neq \emptyset$. By \ref{lemma:characteristics:2} we have $p\subset LB_p$. Note that there exists a ball in $A\cap (LB_p)$ of diameter equal to $R_A(LB_p)-r_A(LB_p)$. Hence,
\begin{align}\label{lemma:characteristics:ineq1}
(R_A(p)-r_A(p))^2 &\leq (R_A(LB_p)-r_A(LB_p))^2 \leq \frac{4}{\pi} m_2( A\cap (LB_p)).
\end{align}
If $p\subset \D(z_0,2t)$, then $\diam p<4t$, so by assumption \ref{lemma:characteristics:2} we have $\diam B_p\leq L\diam p<4Lt$ and $\diam(LB_p)< 4L^2t$. Since $\D(z_0,t)\cap p\neq \emptyset$, we have $\D(z_0,t)\cap (LB_p)\neq \emptyset$, which implies that
$$B_p\subset LB_p\subset \D(z_0,Mt),$$
where $M=1+4L^2$. Thus, 
\begin{align}\label{lemma:characteristics:ineq2}
m_2( A\cap (LB_p))\leq m_2(LB_p)= L^2 \int_{\D(z_0,Mt)} \chi_{B_p}\, dm_2.
\end{align} 
By combining \eqref{lemma:characteristics:ineq1} and \eqref{lemma:characteristics:ineq2}, we conclude that
\begin{align*}
\sum_{\substack{p\in \mathcal C\\ p\cap A\neq \emptyset, \,p\subset \D(z_0,2t)}} (R_A(p)-r_A(p))^2&\leq  \frac{4L^2}{\pi}\sum_{p\in \mathcal C}\int_{\D(z_0,Mt)} \chi_{B_p}\, dm_2 \\
&=  \frac{4L^2}{\pi}\int_{\D(z_0,Mt)} \sum_{p\in \mathcal C} \chi_{B_p}\, dm_2 \leq 4L^3M^2 t^2;
\end{align*}
the last inequality follows from assumption \ref{lemma:characteristics:2}, which implies that $\mathcal C$ is countable and $\sum_{p\in \mathcal C}\chi_{B_p}$ is measurable and bounded above by $L$. Note that
$$w_A(p)=\log \frac{R_A(p)}{r_A(p)}\leq \frac{R_A(p)-r_A(p)}{r_A(p)}\leq 2\frac{R_A(p)-r_A(p)}{t}.$$
We conclude that
\begin{align}\label{lemma:characteristics:ineq3}
\sum_{\substack{p\in \mathcal C\\ p\cap A\neq \emptyset,\,p\subset \D(z_0,2t)}} w_A(p)^2\leq 16L^3M^2.
\end{align}

Suppose that $p\cap A\neq\emptyset $ and $p\not\subset \D(z_0,2t)$. Such components intersect both complementary components of $\mathbb A(z_0;t,2t)$. By assumption \ref{lemma:characteristics:1}, there exist at most $L$ such components $p\in \mathcal C(\Omega)$. Also, we have  $w_A(p)\leq \log 2$ for each $p\in \mathcal C(\Omega)$. Thus,
\begin{align}\label{lemma:characteristics:ineq4}
\sum_{\substack{p\in \mathcal C(\Omega)\\ p\cap A\neq \emptyset,\, p\not\subset \D(z_0,2t)}} w_A(p)^2\leq  L(\log2)^2.
\end{align}
Finally, by assumption \ref{lemma:characteristics:2} the set $\mathcal C(\Omega)\setminus \mathcal C$ has at most $L$ non-degenerate elements.   Since $w_A(p)=0$ for degenerate components $p\in \mathcal C(\Omega)$, we have 
\begin{align}\label{lemma:characteristics:ineq5}
\sum_{p\in \mathcal C(\Omega)\setminus \mathcal C} w_A(p)^2\leq  L(\log2)^2.
\end{align}
By combining \eqref{lemma:characteristics:ineq3}, \eqref{lemma:characteristics:ineq4}, and \eqref{lemma:characteristics:ineq5}, we conclude that
$$\sum_{\substack{p\in \mathcal C(\Omega)\\ p\cap A\neq \emptyset}} w_{A} (p)^2 \leq 16L^3M^2+ 2L(\log2)^2.$$
This completes the proof. 
\end{proof}

We now restate Theorem \ref{theorem:intro} in an expanded form and we give the proof.
\begin{theorem}\label{theorem:john:restate}
Let $\Omega\subset \C$ be a John domain. Then there exists a conformal map $f$ from $\Omega$ onto a circle domain. Moreover for each $q\in \mathcal C(\Omega)$ the set $\hat f(q)$ is a singleton if and only if $q$ is a singleton. 
\end{theorem}

\begin{proof}
By Lemma \ref{lemma:diameter}, for each $j\in \N$ at most finitely many components of $\mathcal C(\Omega)$ have spherical diameter greater than $j^{-1}$. Let $\Omega_j'$ be the finitely connected domain such that $\mathcal C(\Omega_j')$ consists of these components. By the Koebe uniformization theorem, for $j\in \N$ there exists a conformal map $f_j$ from $\Omega_j'$ onto a circle domain $D_j'$, which we normalize so that it fixes three points of $\Omega$. By a normality criterion \cite{LehtoVirtanen:quasiconformal}*{Theorem II.5.1}, and by passing to a subsequence, we may assume that $\{f_j\}_{j\in \N}$ converges locally uniformly in $\Omega$ to a conformal map from $\Omega$ onto a domain $D\subset \widehat \C$. Our goal is to show that $D$ is a circle domain. For this, it suffices to show that Theorem \ref{theorem:general_caratheodory} is applicable for each $q\in \mathcal C(\Omega)$. Indeed, then for any Hausdorff limit $E$ of $\{\hat f_j(q_j)\}_{j\in \N}$, which is necessarily a disk or a point, we have $\hat f(q)=E$, so $\hat f(q)$ is a disk or a point.

We first verify conditions \ref{theorem:general_caratheodory:1} and \ref{theorem:general_caratheodory:3} of Theorem \ref{theorem:general_caratheodory}. Suppose that $q\in \mathcal C(\Omega)$ is non-degenerate, so $q\in \mathcal C(\Omega_j')$ for all sufficiently large $j\in \N$. By ignoring finitely many indices, we may have that $q\in \mathcal C(\Omega_j')$ for all $j\in \N$. We set $\Omega_j=\Omega_j'$, $D_j=D_j'$, and $q_j=q$ for $j\in \N$, and we note that condition \ref{theorem:general_caratheodory:1} of Theorem \ref{theorem:general_caratheodory} holds. Now, suppose that $q\in \mathcal C(\Omega)$ is a singleton, so $q\in \Omega_j'$ for every $j\in \N$. For $j\in \N$ we set $\Omega_j=\Omega_j'\setminus q$ and $D_j=f_j(\Omega_j)=D_j'\setminus f_j(q)$. Thus, condition \ref{theorem:general_caratheodory:1} of Theorem \ref{theorem:general_caratheodory} holds for $\{\Omega_j\}_{j\in \N}$ with $q_j=q$ for each $j\in \N$. Moreover, in both cases the map $\{f_j\colon \Omega_j\to D_j\}_{j\in \N}$ converges locally uniformly in $\Omega$ to $f\colon \Omega\to D$, so condition \ref{theorem:general_caratheodory:3} of Theorem \ref{theorem:general_caratheodory} holds as well.

Proposition \ref{prop:packing} and Lemma \ref{lemma:count} imply that condition \ref{lemma:characteristics:1} of Lemma \ref{lemma:characteristics} holds for $\Omega$. Moreover, Proposition \ref{prop:shadows} shows that condition \ref{lemma:characteristics:2} of Lemma \ref{lemma:characteristics} holds; here $\mathcal C$ is taken to be the collection of non-degenerate components $p\in \mathcal C(\Omega)$ that do not contain $\infty$, so $\mathcal C(\Omega)\setminus \mathcal C$ has at most one non-degenerate element. Thus, the conclusion of Lemma \ref{lemma:characteristics} is true for any annulus $A=\mathbb A(z_0;r/2,r)$. In particular, given that $\mathcal C(\Omega_j)\subset \mathcal C(\Omega)$ for each $j\in \N$, we see that condition \ref{prop:geometric_sufficient:3} of Proposition \ref{prop:geometric_sufficient} holds. Condition \ref{prop:geometric_sufficient:1} of Proposition \ref{prop:geometric_sufficient} also holds, by Lemma \ref{lemma:diameter}. By construction, \ref{prop:geometric_sufficient:2} is true as well. By Proposition \ref{prop:geometric_sufficient}, the condition of Lemma \ref{lemma:potential} holds. Thus, condition \ref{theorem:general_caratheodory:2} of Theorem \ref{theorem:general_caratheodory} is true, as desired. We conclude that $D$ is a circle domain and by Theorem \ref{theorem:koebe_nondegenerate}, $\hat f(q)$ is a point if and only if $q$ is a point.
\end{proof}

\subsection{Necessary and sufficient condition}\label{section:exhaustion}
We restate and prove Theorem \ref{theorem:exhaustion}. Recall the definitions given in Section \ref{section:exhaustion_intro}.

\begin{theorem}
A domain $\Omega\subset \widehat \C$ satisfies Koebe's conjecture if and only if there exist an exhaustion $\{\Omega_j\}_{j\in \N}$ of $\Omega$ and a Jordan region $U\subset \widehat \C$ that contains $\widehat \C\setminus \Omega$ such that for every $q\in \mathcal C(\Omega)$ we have
\begin{align}\label{theorem:exhaustion:claim}
\lim_{\substack{z\to  q\\z\in \Omega}} \limsup_{j\to\infty} \Mod_{\Omega_j} \mathcal G^*(z,q_j(q); U)=0.
\end{align}
In this case, if $f_j$ is a conformal map from $\Omega_j$ onto a circle domain $D_j\subset \widehat \C$, $j\in \N$, such that $\{f_j\}_{j\in \N}$ converges locally uniformly in $\Omega$ to a conformal map $f$ on $\Omega$, then for each $q\in \mathcal C(\Omega)$ we have
$$\lim_{j\to\infty}\hat f_j(q_j(q))=\hat f(q).$$
\end{theorem}

\begin{proof}
Suppose that there exists an exhaustion of $\Omega$ as in the statement. By Koebe's uniformization theorem, there exists a sequence of conformal maps $f_j$ from $\Omega_j$ onto a circle domain $D_j$, $j\in \N$. By appropriate normalizations, we may assume that there exist three points of $\Omega$ that are fixed by $f_j$ for each $j\in \N$. Thus, by a normality criterion \cite{LehtoVirtanen:quasiconformal}*{Theorem II.5.1}, there exists a subsequence of $\{f_j\}_{j\in \N}$ that converges locally uniformly to a conformal map $f$ from $\Omega$ onto a domain $D$.  Note that for each $q\in \mathcal C(\Omega)$ the set $\hat f_j(q_j(q))$ is a disk for each $j\in \N$. Thus, all Hausdorff limits of $\{\hat f_j(q_j(q))\}_{j\in \N}$ are disks or points. Let $E$ be such a Hausdorff limit. By Theorem \ref{theorem:general_caratheodory}, the set $\hat f(q)$ is the complement of the component of $\widehat \C\setminus E$ that contains $D$. Hence, $\hat f(q)=E$, the set $\hat f(q)$ is a disk or a point, and
$$\lim_{j\to\infty}\hat f_j(q_j(q)) =\hat f(q).$$
Therefore, $D$ is a circle domain. Thus, Koebe's conjecture is true for $\Omega$. Note that the last part of the theorem follows from the above argument. 

Conversely, suppose that Koebe's conjecture is true for $\Omega$, so there exists a conformal map $f$ from $\Omega$ onto a circle domain $D$. Without loss of generality, suppose that $\infty\in \Omega$ and $f(\infty)=\infty$. By \cite{NtalampekosRajala:exhaustion}*{Theorem 2.1} there exists an exhaustion $\{D_j\}_{j\in \N}$ of $D$ that is $K_0$-\textit{quasiround} for $K_0=43$, meaning that for each $j\in \N$ and $q\in \mathcal C(D_j)$ there exists a ball $B_q$ in $\C$ such that
\begin{align}\label{theorem:exhaustion:quasiround}
B_q\subset q\subset K_0B_q.
\end{align}
Let $\{\Omega_j\}_{j\in \N}$ be the preimage of this exhaustion under $f$. We show \eqref{theorem:exhaustion:claim} for this exhaustion.  Consider a ball $B$ in $\C$ such that $B\supset \widehat \C\setminus D$ and let $U$ be the Jordan region that is bounded by $f^{-1}(\partial B)$ and does not contain $\infty$. By the conformal invariance of transboundary modulus, it suffices to prove that 
$$\lim_{\substack{w\to q\\w\in D}} \limsup_{j\to\infty} \Mod_{D_j} \mathcal G^*(w,q_j(q);B) =0$$
for every $q\in \mathcal C(D)$. That is, we wish to verify condition \ref{theorem:general_caratheodory:2} of Theorem \ref{theorem:general_caratheodory} for the domain $D$, the component $q\in \mathcal C(D)$, the sequence $\{D_j\}_{j\in \N}$, and the components $q_j=q_j(q)\in \mathcal C(D_j)$. It suffices to verify the condition of Lemma \ref{lemma:potential}. For this, it is enough to establish condition \ref{prop:geometric_sufficient:3} of Proposition \ref{prop:geometric_sufficient} for $A_r=\mathbb A(z_0;r/2,r)$, where $z_0\in \partial q$ and $r>0$. 

We verify the assumptions of Lemma \ref{lemma:characteristics} for the annulus $A_r$ and the domain $D_j$ when $j$ is large enough. If a ball $B$ intersects both complementary components of $A_{2r}$, then there exists a ball $B'\subset B\cap A_{2r}$ of radius $r/2$. Thus, $m_2(A_{2r}\cap B)\geq \pi r^2/4$. Hence, if a collection of $N$ disjoint balls intersect both complementary components of $A_{2r}$, then
$$N \frac{\pi r^2}{4} \leq m_2(A_{2r})= 3\pi r^2.$$
That is, at most $12$ elements of $\mathcal C(D)$ can intersect both boundary components of $A_{2r}$. Since $\{D_j\}_{j\in \N}$ is an exhaustion of $D$, we conclude that for all sufficiently large $j\in \N$ at most $12$ components $p\in \mathcal C(D_j)$ intersect both boundary components of $A_{2r}$. Thus, condition \ref{lemma:characteristics:1} of Lemma \ref{lemma:characteristics} holds for the domain $D_j$ for all sufficiently large $j\in \N$ with constant $L=12$. Also, condition \ref{lemma:characteristics:2} of Lemma \ref{lemma:characteristics} holds automatically with constant $K_0$ by \eqref{theorem:exhaustion:quasiround}. By Lemma \ref{lemma:characteristics}, we conclude that
$$\limsup_{j\to\infty}\sum_{\substack{p\in \mathcal C(D_j)\\ p\cap A_r\neq \emptyset}} w_{A_r} (p)^2 \leq C$$
for a universal constant $C>0$, independent of $z_0$ and $r$. Therefore,
$$\liminf_{r\to 0^+}\limsup_{j\to\infty}\sum_{\substack{p\in \mathcal C(D_j)\\ p\cap A_r\neq \emptyset}} w_{A_r} (p)^2 <\infty.$$
This verifies condition \ref{prop:geometric_sufficient:3} of Proposition \ref{prop:geometric_sufficient}, as desired.
\end{proof}

\bibliography{../../../biblio} 
\end{document}